\documentclass[12pt]{amsart}

\usepackage{amsthm,amssymb,amstext,amscd,amsfonts,amsbsy,amsxtra,latexsym,amsmath}
\usepackage{fullpage}
\usepackage[english]{babel}
\usepackage[latin1]{inputenc}
\usepackage{verbatim}
\usepackage{graphicx} 
\numberwithin{figure}{section}

\newcommand{\field}[1]{\mathbb{#1}}
\newcommand{\N}{\field{N}}
\newcommand{\Z}{\field{Z}}
\newcommand{\R}{\field{R}}
\newcommand{\C}{\field{C}}

\renewcommand{\H}{\mathbb{H}}

\newtheorem{theorem}{Theorem}[section]

\newtheorem{lemma}[theorem]{Lemma}
\newtheorem{corollary}[theorem]{Corollary}
\newtheorem{proposition}[theorem]{Proposition}

\newtheorem*{theorem*}{Theorem}
\newtheorem{remark}[theorem]{Remark}

\numberwithin{equation}{section}

\renewenvironment{proof}[1][Proof]{\begin{trivlist}
\item[\hskip \labelsep {\bfseries #1:}]}{\qed\end{trivlist}}

\newcommand{\bea}{\begin{eqnarray}} 
\newcommand{\eea}{\end{eqnarray}} 
\newcommand{\be}{\begin{equation*}} 
\newcommand{\ee}{\end{equation*}} 
\newcommand{\benn}{\begin{equation}} 
\newcommand{\eenn}{\end{equation}} 

\thanks{The author's research is supported by the DFG-Graduiertenkolleg 1269  ''Globale Strukturen in Geometrie und Analysis'' \\
This result in this paper are partially contained in the author's diploma thesis \cite{Za} written under the supervision of Prof. Dr. Kathrin Bringmann at the University of Cologne and the University of Bonn}
\begin{document}

\title[]{Asymptotics of crank generating functions and Ramanujan congruences}
\author{Jose Miguel Zapata Rolon$^1$}
\address{$^1$ Mathematical Institute\\University of
Cologne\\ Weyertal 86-90 \\ 50931 Cologne \\Germany}
\email{rzapata@math.uni-koeln.de}

\begin{abstract}
In this paper we obtain asymptotic formulas for the Fourier coefficients of an infinite family of crank generating functions. Moreover we use this result to show that the crank obeys certain inequalities. This implies that the crank can not explain any partition congruences in the usual way beside the three deduced by Ramanujan.
\end{abstract}

\maketitle
\section{Introduction and Statement of the results}
The theory of partitions is an intriguing example for the interplay between number theory and analytic methods. The important question how many integer partitions does a non-negative integer have was answered asymptotically by Hardy and Ramanujan \cite{HR} using the Circle Method. Improving these results Rademacher obtained the following formula \cite{R}
\begin{align}\label{part}
p(n) = \frac{ 2 \pi  }{\left(24n -1\right)^{\frac{3}{4}}}\sum_{k=1}^{\infty} \frac{A_k(n)}{k} I_{\frac{3}{2}}\left(\frac{\pi\sqrt{24n-1}}{6k}\right),
\end{align}
were $A_k(n)$ are certain type of Kloosterman sums and $I_{\frac{3}{2}}$ is the order $3/2$ Bessel function.
Ramanujan found in \cite{Ram2} that the integer partition function fulfills interesting congruences and proved these congruences by anticipating the general theory of Hecke operators and $l$-adic modular forms. Nevertheless, his approach gave little combinatorial insight  why $p(n)$ fulfills
\begin{eqnarray}\label{racon}
 p(5n+4)&\equiv & 0 \pmod{5}, \nonumber\\
 p(7n+5)&\equiv & 0  \pmod{7} , \\
 p(11n+6)&\equiv& 0  \pmod{11} \nonumber.
\end{eqnarray}
Some 25 years later, Dyson introduced \cite{Dy1} a combinatorial statistic, called the \emph{rank} of a partition $\lambda$, that partially explained the observed congruences,
\begin{equation}
\textrm{rank}(\lambda) := \textrm{largest part of}\, \lambda - \textrm{number of parts of} \, \lambda .
\end{equation}
Moreover he conjectured the existence of another statistic that explains all congruences simultaneously from a combinatorial point of view and called it the \emph{crank}. The definition was given forty years later by Andrews and Garvan in \cite{AG1}, where they completed the search for combinatorial decompositions of the three congruences using previous results of Garvan \cite{G1}.
Denote by $o(\lambda)$ the number of ones in  $\lambda$, and $\mu(\lambda)$ the number of parts strictly larger than $o(\lambda)$, then
\begin{equation}
\textrm{crank}(\lambda):=\begin{cases}  \textrm{largest part of}\, \lambda \qquad &\textrm{if}\, o(\lambda) = 0, \\
					\mu(\lambda) - o(\lambda)	   \qquad &\textrm{if}\, o(\lambda) > 0. 	
                         \end{cases}
\end{equation}
Moreover let $M(m,n) (\textrm{resp.}\, N(m,n))$ be the number of partitions of $n$ with crank (resp. rank) $m$. The two-variable generating function may be written as \cite{AG1,Dy1,G1}
\begin{align}\label{pochh}
 C(x;q) &= \sum_{ m \in \Z \atop{n \geq 0}}M(m,n)x^m q^n = \frac{1-x}{(q)_{\infty}}\sum_{n\in \Z}\frac{(-1)^n q^{\frac{n(n+1)}{2}}}{1-xq^n} = \frac{(q)_{\infty}}{(xq)_{\infty}(x^{-1}q)_{\infty}}, \\
 R(x;q) &= \sum_{ m \in \Z \atop{n \geq 0}}N(m,n)x^m q^n = \frac{1-x}{(q)_{\infty}}\sum_{n\in \Z}\frac{(-1)^n q^{\frac{n(3n+1)}{2}}}{1-xq^n}, \nonumber
\end{align}
where the \emph{q-Pochhammer} symbol is given for $ n \in \N_{0} \cup \lbrace \infty \rbrace$ by $(a)_{n}:=(a;q)_{n} := \prod_{i=0}^{n-1}\left( 1-a q^{i} \right)$.
Although the generating functions of crank and rank have a very similar shape and are defined by the same motivation it turns out that they have different automorphic behavior. See Remark \ref{remarktrafo}. 
In this paper we want to compute the asymptotic values of the crank generating function. \\
Throughout, let $z \in \C$ with $\textrm{Re}(z)>0$ and $0 \leq h < k$ with $(h,k)=1$. Let $x = e^{2\pi i u}$ and $q= e^{- 2 \pi z}$. Let $h^{\prime}$ be a solution to the congruence $hh^{\prime} \equiv -1 \pmod{k}$ if $k$ is odd and let $h^{\prime}$ be a solution to the congruence $hh^{\prime} \equiv -1 \pmod{2k}$ if $k$ is even. Let $0 < l < c$, be the unique solution to $ l \equiv ak \pmod{c}$. Finally let $0 < a <c $ be coprime integers with $c$ odd. Moreover we make the technical assumption that $c$ is prime. For a more general treatment ($c$ odd)
see \cite{Za}.
To state the theorem, we have to fix more notation and define the following sum for $m,n \in \Z$ and $ c \mid k$: 
\begin{equation}\label{BKLO}
\widetilde{B}_{a,c,k} \left(n,m\right) :=(-1)^{ak+1} \sin \left(\frac{\pi a}{c} \right)\sum_{h \pmod{k}^{*}} \frac{\omega_{h,k}
 }{\sin \left(\frac{\pi a h'}{c} \right)}  e^{-\frac{ \pi i a^2 k h'}{c^2}}  
e^{\frac{2 \pi i}{k}\left(nh+mh'\right)}.
\end{equation}
Here the sum  runs over all primitive residue classes modulo $k$ and this summation is denoted by $ h \pmod{k}^{*}$.  For the case $c \nmid k$ we define
\begin{equation}\label{DKLO} 
D_{a,c,k}\left(m,n\right) = (-1)^{ak + l} \sum_{h \pmod{k}^{*}} \omega_{h,k} e^{\frac{2\pi i}{k}\left( n h + m h^{\prime}\right)}.
\end{equation}
Moreover we define:
\begin{eqnarray}\label{delta}
\delta_{a,c,k,r}^{i} : =
\left\{ 
\begin{array}{ll}
-\left( \frac{1}{2}+r \right)\frac{l}{c} 
+\frac{1}{2} \left( \frac{l}{c}\right)^2 +\frac{1}{24} & \text{if } i = +,\\[1ex]
\frac{l}{2c} + \frac{1}{2}\left(\frac{l}{c} \right)^2 -\frac{23}{24} 
- r\left(1-\frac{l}{c} \right)&\text{if } i = -,
\end{array}
\right.
\end{eqnarray}
and 
\begin{align}\label{m}
m^{+}_{a,c,k,r}  := & \frac{1}{2c^{2}}\left(-a^2 k^2 + 2 l a k - a k c -l^2 + l c - 2a r k c + 2 l c r \right), \\
m^{-}_{a,c,k,r}  := & \frac{1}{2c^{2}}\left(-a^2 k^2 + 2 l a k - a k c -l^2 + 2 c^2 r - 2 l r c + 2 a r k c + 2 l c + 2 c^2 - a k c \right) \nonumber. 
\end{align}
Then we have the following result:
\begin{theorem}\label{main1}
Let $C\left(e^{\frac{2 \pi i a}{c}}; q \right)=: 1 + \sum_{n=1}^{\infty}\widetilde{A}\left( \frac{a}{c};n\right)q^n$ and $\varepsilon > 0$. If $0<a<c$ are coprime integers, $c$ is odd and $n$ is a positive integer,  then we have 
\begin{multline*}
  \widetilde{A}\left(\frac{a}{c};n\right) = 
     \frac{4 \sqrt{3} i }{  \sqrt{24n-1}} \sum_{1 \leq k \leq \sqrt{n} \atop c|k} 
\frac{\widetilde{B}_{a,c,k}(-n,0)}{\sqrt{k}}  
\sinh \left(\frac{\pi  \sqrt{24n-1} }{6k}\right) + \frac{ 8 \sqrt{3}   \cdot   \sin \left(\frac{\pi a}{c} \right) }{\sqrt{24n-1}} \\ \times
\sum_{1 \leq k\leq \sqrt{n}\atop {c \nmid k\atop {r \geq 0  \atop {\delta^{i}_{a,c,k,r}>0 \atop { i \in \lbrace +,- \rbrace}}} } }
\frac{D_{a,c,k}(-n,m_{a,c,k,r}^{i})}{\sqrt{k}}   
 \sinh \left( 
 \frac{\pi\sqrt{2 \delta_{a,c,k,r}^{i}(24n-1)}}{\sqrt{3}k}
 \right)
 +O \left( n^{\varepsilon}\right).
\end{multline*}
\end{theorem}
Define $M(a,c,n)(\textrm{resp.} \, N(a,c,n))$ to be the number of partitions of $n$ with crank(resp. rank) equal to $a$ modulo $c$. As a Corollary of Theorem \ref{main1} we can give asymptotic values of $M(a,c,n)$: 
\begin{corollary}
Let $\varepsilon > 0$ and $0 \leq a < c$ with $c$ an odd integer. Let $n$ be a positive integer, then we have:
\begin{align*}
M(a,c;n) = & \frac{2 \pi}{c \sqrt{24n-1}}\sum_{k=1}^{\infty}\frac{A_k(n)}{k}\, I_{\frac{3}{2}}\left(\frac{\pi \sqrt{24n-1}}{6k}\right) \\ +
 & \frac{1}{c}\sum_{j=1}^{c-1}\zeta_{c}^{-aj}\left(\frac{4\sqrt{3}i}{\sqrt{24n-1}}\right.\sum_{c \mid k} \frac{\widetilde{B}_{j,c,k}(-n,0)}{\sqrt{k}}\sinh\left(\frac{\pi}{6k}\sqrt{24n-1}\right) \\ &
+\frac{8 \sqrt{3}\sin\left(\frac{\pi j}{c}\right)}{\sqrt{24n-1}}\sum_{k,r \atop { c\nmid k \atop { \delta^{i}_{j,c,k,r} >0 \atop { i \in \lbrace +,- \rbrace}}}} \frac{D_{j,c,k}(-n, m_{j,c,k,r}^{i})}{\sqrt{k}}\left.\sinh\left(\sqrt{\frac{2 \delta_{j,c,k,r}^{i}(24n-1)}{3}}\frac{\pi}{k}\right)\right) \\ & + O(n^{\varepsilon}).
\end{align*}
\end{corollary}
In \cite{BB1}, the following theorem was shown. 
\begin{theorem*}
 Let  $ 0 \leq a < b \leq \frac{c-1}{2} $ and let $c> 9$ be an odd integer, then we have for $n > N_{a,b,c}$, where $N_{a,b,c}$ is an explicit constant, the inequality:
\be
N(a,c;n) > N(b,c;n).
\ee
\end{theorem*}
This shows that the rank can not group the different partitions into equally sized congruence classes for $c=11$ and so gives an explanation why the rank fails to explain the congruence for $c=11$.
Here we use Theorem \ref{main1} to obtain the analog inequality for the crank:
\begin{theorem}\label{blabla}
Let  $ 0 \leq a < b \leq \frac{c-1}{2} $ and let $c> 11$ be an odd integer, then we have for $n > N_{a,b,c}$, where $N_{a,b,c}$ is an explicit constant, the inequality:
\be
M(a,c;n) > M(b,c;n).
\ee
\end{theorem}
We can also characterize the inequalities for $c \leq 11$. This is done in Corollary \ref{rest}. \\ Apart from the Ramanujan congruences, there are many other partition congruences for higher moduli, e.g 
\begin{equation}\label{hicon}
p(11^3 \cdot 13 n +237) \equiv 0 \pmod{13}, 
\end{equation} 
see for example the work of S. Ahlgren and K. Ono \cite{AO}. Using the inequalities of Theorem \ref{blabla} the next result is immediate. 
\begin{corollary}\label{corblabla}
 The crank can not explain any partition congruences other than (\ref{racon}). 
\end{corollary}
%Many other like Atkin, or Ahlgren and Ono \cite{AO} constructed many other partition congruences. As an example see 
%\be
%p(11^3 \cdot 13 n + 237)\equiv 0 \pmod{13}.
%\ee
By Corollary \ref {corblabla} the congruence (\ref{hicon}) can not be explained by the crank.\\
The remainder of the paper is as follows. In Section 2 we show a transformation formula for an infinite family of crank generating functions. Next, in Section 3 we prove certain bounds of Kloosterman sums and use these for the Circle Method to compute the Fourier coefficients of this family. In Section 4 we show that the crank obeys certain inequalities by bounding all the error terms in the Circle Method explicitly.
% \section*{Acknowledgements}
% The author would like to thank Prof. Dr. Kathrin Bringmann for suggesting this topic for the diploma thesis and many helpful discussions.  He also thanks Dr. Ben Kane for many helpful discussions and Michael Mertens for helpful comments on the paper. 
\section{A transformation formula from classical modular forms}
In this section, we prove investigate the transformation behavior of the crank generating function under the action of $\mathrm{SL}_2(\Z)$. First, we introduce the functions we need. The Dedekind $\eta$-function is defined as
\be
\eta(\tau) := e^{\frac{\pi i \tau}{12}}\prod_{n=0}^{\infty}\left(1-e^{2\pi i n\tau}\right),
\ee
and the Jacobi $\vartheta$-function is defined as
\be 
\vartheta(u;\tau):= \sum_{\nu \in \Z + \frac{1}{2}}e^{\pi i \nu^2 \tau + 2\pi i \nu \left(u +\frac{1}{2}\right)}.
\ee
Here we collect the necessary transformation formulas of $\eta$ and $\vartheta$. First, we introduce the needed quantities to state the transformation rules. Therefore define
\begin{equation}\label{chi}
\chi(h,h^{\prime},k) := i^{-\frac{1}{2}}\omega_{h,k}^{-1}e^{-\frac{\pi i}{12 k}\left(h^{\prime}-h\right)}.
\end{equation}
Here $h^{\prime}$ is a solution to $h h^{\prime} \equiv -1 \pmod{k}$ and
\begin{equation}\label{omega} 
\omega_{h,k} := \exp\left(\pi i s\left(h,k\right)\right),
\end{equation}
where the \emph{Dedekind sums} $s(h,k)$ are explicitly given by
\be 
s(h,k):= \sum_{ \mu \pmod{k}}\left(\left(\frac{\mu}{k}\right)\right)\left(\left(\frac{h \mu}{k}\right)\right).
\ee
In the above, the saw tooth function is defined by
\begin{eqnarray*}
\left(\left(x\right)\right) : =
\begin{cases}
x - \lfloor x \rfloor - \frac{1}{2}  &\text{if } x \in \R \setminus \Z, \\[1ex]
0              &\text{if } x \in \Z.
\end{cases}
\end{eqnarray*}
For $z \in \C$ with Re$(z)>0$ we have
\begin{equation}
  \label{eta}
\eta\left(\frac{h +iz}{k}\right)=\sqrt{\frac{i}{z}}\chi\left(h,h^{\prime},k\right)\eta\left( \frac{h^{\prime} + \frac{i}{z}}{k}\right),
\end{equation}
where we take the principal branch of the square root. Moreover, $\eta$ is a modular form of weight $\frac{1}{2}$ with multiplier system.
Now we can also state the transformation formula for the Jacobi $\vartheta$-function \cite{BMR, Zw}.
Define $x := e^{2 \pi i \omega}$ and $q := e^{2\pi i \tau}$ where $\omega \in \C$ and $\tau \in \H$. Let $h,k$ be coprime integer with $h^{\prime}$ like above. Then $\vartheta$ satisfies the
%\item $\vartheta(\omega +1;\tau) = -\vartheta(\omega;\tau)$.
%\item $\vartheta(\omega + \tau;\tau) = -e^{-\pi i \tau - 2\pi i \omega}\vartheta(\omega;\tau)$,
%\item Up to a multiplicative constant, $\omega \mapsto \vartheta(\omega;\tau)$ is the unique holomorphic function satisfying (1), (2).
%\item $\vartheta(-\omega;\tau)= -\vartheta(\omega;\tau).$
%\item The zeros of $\vartheta$ are the points $\omega = n \tau + m$, with $m,n \in \Z$. These are all simple zeros.
%\item $\vartheta(\omega;\tau +1)= e^{\frac{\pi i}{4}}\vartheta(\omega ; \tau)$,
%\item $\vartheta\left(\frac{\omega}{\tau};-\frac{1}{\tau}\right)= - i \sqrt{-i\tau}e^{\frac{\pi i \omega^2}{\tau}}\vartheta(\omega;\tau)$, 
Jacobi triple product identity 
\begin{equation}\label{JTPI}
 \vartheta(\omega;\tau) = - 2\sin(\pi \omega) q^{\frac{1}{8}}(q)_{\infty}(xq)_{\infty}(x^{-1}q)_{\infty},
\end{equation}
and if Re$(z)>0 $, then 
\begin{equation}\label{thtrafo}
\vartheta\left(\omega;\frac{h +iz}{k}\right)= \chi^{3}\sqrt{\frac{i}{z}}e^{-\frac{\pi k \omega^2}{z}}\vartheta\left(\frac{i\omega}{z}; \frac{h^{\prime} + \frac{i}{z}}{k}\right).
\end{equation}
\begin{remark}
Using the Jacobi triple product formula (\ref{JTPI}) and the Pochhammer symbol representation (\ref{pochh}) of the crank generating function
%\benn
%C(x;q) = \frac{(q)_{\infty}^{2}}{(xq)_{\infty} (x^{-1}q)_{\infty}(q)_{\infty}}. 
%\eenn
it is easy to see that 
%The Jacobi theta function obeys the Jacobi triple product identity (see Theorem %\ref{sander}
%)
%\benn 
%\vartheta(u;iz)  = - 2 \sin(\pi u) q^{\frac{1}{8}}(q)_{\infty}(xq)_{\infty}(x^{-1}q)_{\infty},
%\eenn 
%and the Dedekind $\eta$-function can be expressed in the following way
%\benn 
%\eta(iz) = q^{\frac{1}{24}}(q)_{\infty}.
%\eenn 
%Plugging in the $\eta^2$ and the Jacobi tripel product identity we arrive at:
\begin{equation} \label{cranketa}
C\left(e^{2\pi i u};e^{ -2\pi z} \right) = \frac{-2 \sin(\pi u) q^{\frac{1}{24}} \eta^2(iz)}{\vartheta(u;iz)}.
\end{equation}
\end{remark}
Using these results for $\eta$ and $\vartheta$ we obtain: 
\begin{proposition}\label{trafo}
We define $q_1:= e^{\frac{2 \pi i}{k}\left(h^{\prime}+ \frac{i}{z} \right)}$. Then the following is true:
\begin{itemize}
\item[(1)] For $c \mid k$ we have
\begin{multline*}
C\left(e^{\frac{2 \pi i a}{c}}; e^{\frac{2 \pi i }{k}\left(h + iz\right)}\right) = \frac{i \sin\left(\frac{\pi a}{c}\right)}{z^{\frac{1}{2}}\sin\left(\frac{\pi a h^{\prime}}{c}\right)} (-1)^{ak + 1} \omega_{h,k} e^{\frac{\pi}{12k}\left(z^{-1}-z\right)- \frac{\pi i a^2 k h^{\prime}}{c^2}} \\ \times C\left(e^{\frac{2 \pi i a h^{\prime}}{c}};e^{\frac{2 \pi i }{k}\left(h^{\prime} + \frac{i}{z}\right)}\right). 
\end{multline*}
\item[(2)] For $ c \nmid k$ we have
\begin{multline*} 
C\left(e^{\frac{2 \pi i a}{c}}; e^{\frac{2 \pi i }{k}(h + iz)}\right)= \frac{4 i \sin( \frac{\pi a}{c}) \omega_{h,k}(-1)^{ak + l+1}}{z^{\frac{1}{2}}} e^{-\frac{ \pi a^2 h^{\prime} k}{c^2} + \frac{2 \pi i h^{\prime} l a} {c^2}} \\ \times q_1^{-\frac{l^2}{2 c^2}}e^{\frac{\pi}{12k}(z^{-1}-z)} C\left(a h^{\prime}, l, c ; q_1\right),
\end{multline*}
\end{itemize}  
where we have defined
\be 
C(a,b,c;q):= \frac{i}{2(q)_{\infty}}\left( \sum_{m=0}^{\infty}\frac{(-1)^m e^{-\frac{\pi i a}{c}}q^{\frac{m(m+1)}{2} +\frac{b}{2c}}}{1 - e^{-\frac{2\pi i a}{c}}q^{m + \frac{b}{c}}} - \sum_{m=1}^{\infty}\frac{(-1)^m e^{\frac{\pi i a}{c}}q^{\frac{m(m+1)}{2} -\frac{b}{2c}}}{1 - e^{\frac{2\pi i a}{c}}q^{m - \frac{b}{c}}} \right).
\ee
\end{proposition}
\begin{proof}
(1). Using (\ref{eta}),(\ref{cranketa}), %\ref{eta} 
and (\ref{thtrafo}), %\ref{sander}
 we obtain %Remember that
%\be
%\eta\left(\frac{h +iz}{k}\right)=\chi \sqrt{\frac{i}{z}} \eta\left( \frac{h^{\prime} + \frac{i}{z}}{k}\right),
%\ee
%where $\chi = i^{-\frac{1}{2}} \omega_{h,k}^{-1} e^{-\frac{\pi i}{12 k}\left(h - h^{\prime}\right)}$, and
%\be 
%\vartheta\left(u,\frac{h +iz}{k}\right)= \chi^{3}\sqrt{\frac{i}{z}}e^{-\frac{\pi k u^2}{z}}\vartheta\left(\frac{iu}{z}, \frac{h^{\prime} + \frac{i}{z}}{k}\right).  
%\ee
%The $\omega_{h,k}$ are the multipliers occurring in the transformation formula of the partition function. From this the following identity is deduced:
\be 
C\left(e^{2 \pi i u}; e^{\frac{2 \pi i}{k}(h + iz)}\right) = - \frac{2 \sin(\pi u)i}{z^{\frac{1}{2}}}\omega_{h,k} e^{\frac{\pi i}{12k}\left( h^{\prime} - h\right)} e^{\frac{\pi i}{12 k}\left(h + iz\right)} \frac{\eta^2\left(\frac{1}{k}\left( h^{\prime} + \frac{i}{z}\right)\right)}{\vartheta\left( \frac{iu}{z}; \frac{1}{k}\left(h^{\prime} + \frac{i}{z}\right)\right)} e^{\frac{\pi k u^2}{z}},
\ee
where the $\omega_{h,k}$ are as in (\ref{eta}).
We now assume that $ c \mid k$ and define $A := \frac{ak}{c} \in \Z$ and write $u=\frac{a}{c}$. First of all we replace $\eta^2$ in the numerator by rewriting (\ref{cranketa}):
\begin{multline}\label{n1}
C\left(e^{2 \pi i \frac{a}{c}}; e^{\frac{2 \pi i}{k}(h + iz)}\right) = i \frac{ \sin( \frac{\pi a}{c})}{\sin( \frac{\pi a h^{\prime}}{c})}\omega_{h,k} e^{\frac{\pi }{12k}(z^{-1} - z) + \frac{\pi a^2 k}{z c^2}}C\left(e^{ \frac{2 \pi i a h^{\prime}}{c}}; e^{\frac{2 \pi i}{k}\left(h^{\prime} + \frac{i}{z}\right)}\right) \\ \times \frac{\vartheta\left(\frac{a h^{\prime}}{c}; \frac{1}{k}\left(h^{\prime} + \frac{i}{z}\right)\right)}{\vartheta\left(\frac{ i a}{z c}; \frac{1}{k}\left(h^{\prime} + \frac{i}{z}\right)\right)}.
\end{multline}
The fraction of theta functions can be simplified by noting that the elliptic variable of the numerator function can be seen as a shift of the elliptic variable of denominator function by $A \cdot \tau$ where $\tau$ is the modular variable defined as $\tau := \frac{1}{k}(h^{\prime} + \frac{i}{z})$. This allows to simplify the quotient of $\vartheta$-functions:
\be
\frac{\vartheta(\frac{a h^{\prime}}{c}; \frac{1}{k}(h^{\prime} + \frac{i}{z}))}{\vartheta(\frac{ i a}{z c}; \frac{1}{k}(h^{\prime} + \frac{i}{z}))}  = (-1)^{ak +1 } e^{-\frac{\pi i a^2 k h^{\prime}}{c^2}} e^{\frac{\pi a^2 k}{z c^2}} e^{-\frac{2 \pi a^2 k}{z c^2}}.
\ee
Inserting this expression into the equation \eqref{n1} 
yields the transformation formula for the case $ c \mid k$.\newline 
(2). The case $ c \nmid k $ is more difficult, because in general we have $ \frac{ak}{c} \notin \Z$. Again  using the transformation rules of $\vartheta$ and $\eta$, we obtain:
\begin{equation}\label{crank1} 
C\left(e^{\frac{2 \pi i a }{c}}; e^{\frac{2 \pi i}{k}(h + iz)}\right) = - \frac{2 \sin(\pi u)i}{z^{\frac{1}{2}}}\omega_{h,k} e^{\frac{\pi i}{12k}( h^{\prime} + iz)} \frac{\eta^2\left(\frac{1}{k}\left( h^{\prime} + \frac{i}{z}\right)\right)}{\vartheta\left( \frac{ia}{zc}; \frac{1}{k}\left(h^{\prime} + \frac{i}{z}\right)\right)} e^{\frac{\pi k a^2}{zc^2}}.
\end{equation}
Since we defined $l$ to be the solution to the congruence condition $l \equiv a k  \pmod{c}$ by definition it is clear that $ B:=\frac{l- ak}{c} \in \Z$. We may shift the theta function in the elliptic variable by the modular variable multiplied by $ B$, where the modular variable is given above. We compute the shifted theta function using Lemma 5.3 in \cite{Rh}
and obtain
\be 
\vartheta \left( \frac{ia}{zc} ; \tau \right) = (-1)^{ak + l} e^{\frac{\pi i \left(l-ak_1\right)^2}{c_1^2}\tau}e^{2 \pi i\left(\frac{ l -a k_1}{c_1}\right)\frac{ia}{zc}} \vartheta\left(\frac{- a h^{\prime}}{c} + \frac{l}{c_1}\tau ; \tau \right).
\ee
Inserting this expression for the theta function into \eqref{crank1} yields:
\begin{equation*}
\frac{C\left(e^{2 \pi i \frac{a}{c}}; e^{\frac{2 \pi i}{k}\left(h + iz\right)}\right)}{(-1)^{ak +l +1}} = 2 \frac{\sin(\frac{\pi a}{c})}{z^{\frac{1}{2}}}i \omega_{ h,k} e^{\frac{\pi}{12k}\left(h^{\prime}+\frac{i}{z} \right)}q_1^{-\frac{l^2}{2 c^2}}e^{\frac{2 \pi i l a h^{\prime}}{c^2}- \frac{\pi i a^2 k h^{\prime}}{c^2}} \frac{\eta^2(\tau)}{\vartheta\left(-\frac{a h^{\prime}}{c} + \frac{l}{c}\tau; \tau\right)}.
\end{equation*}
Now replacing the quotient $\frac{\eta^2}{\vartheta}$ by the crank generating function we get:
\begin{equation} \label{cr1} 
\frac{C\left(e^{2 \pi i \frac{a}{c}}; e^{\frac{2 \pi i}{k}\left(h + iz\right)}\right)}{(-1)^{ak +l}} =  \frac{\sin(\frac{\pi a}{c})}{z^{\frac{1}{2}}}i \omega_{h,k} e^{\frac{\pi}{12k}\left(z^{-1}-z\right)}q_1^{-\frac{l^2}{2 c^2}}e^{\frac{2 \pi i l a h^{\prime}}{c^2}- \frac{\pi i a^2 k h^{\prime}}{c^2}} \frac{C\left(e^{-\frac{2 \pi i a h^{\prime}}{c} +\frac{2 \pi i l}{c}\tau} ; e^{2\pi i \tau}\right)}{\sin\left(-\frac{\pi a h^{\prime}}{c} + \frac{\pi l}{c}\tau \right)}.
\end{equation}
We define $x := e^{-\frac{2 \pi i a h^{\prime}}{c} + \frac{2 l \pi i}{c}\tau}$, and use the exponential representation of the sine to deduce
\be
\frac{C\left(e^{-\frac{2 \pi i a h^{\prime}}{c} +\frac{2 \pi i l}{c}\tau} ; e^{2\pi i \tau}\right)}{\sin\left(-\frac{\pi a h^{\prime}}{c} + \frac{\pi l}{c}\tau \right)}= \frac{-2i x^{\frac{1}{2}}}{(q_1)_{\infty}}\sum_{m\in \Z} \frac{(-1)^m q_1^{\frac{m(m+1)}{2}}}{1-xq_1^{m}}. 
\ee
Inserting this expression into \eqref{cr1}, we arrive at
\begin{multline*}
C\left(e^{2 \pi i \frac{a}{c}}; e^{\frac{2 \pi i}{k}(h + iz)}\right) =  \frac{(-1)^{ak +l +1}\sin\left(\frac{\pi a}{c}\right)}{z^{\frac{1}{2}}}\\ \times i \omega_{h,k}e^{\frac{\pi}{12k}\left(z^{-1}-z\right)}q_1^{-\frac{l^2}{2c^2}}e^{\frac{2\pi i l a h^{\prime}}{c} - \frac{\pi i a^2 k h^{\prime}}{c^2 }} C\left(ah^{\prime},l,c;q_1\right).
\end{multline*}
This completes the proof of the transformation formula.
\end{proof}
\begin{remark}\label{remarktrafo}
It is interesting to see that the transformation formula is similar to the rank case. The main and important difference is that we have no mock part and that the step function $s$ does not appear (see \cite{B}). This reveals the fact that the rank and the crank generating functions look very similar but have completely different behavior under the action of the modular group. In particular the crank generating function is a Jacobi form for the full modular group, where the rank is a Mock Jacobi form. For more on Jacobi forms see \cite{EZ}.
\end{remark}
\section{Circle method}
In this section we give a proof of our asymptotic formula for the coefficients of the crank generating function.  Firstly, an important lemma (compare Lemma 3.2 in \cite{B}) is established that is needed to bound certain terms:
\begin{lemma}\label{Lem1}
Let $n, m, k,D \in \Z$ with $(D,k)=1 $, $0\leq \sigma_1 < \sigma_2 \leq k$ and let $\varepsilon > 0$. 
\begin{itemize}
\item[(1)] We have 
\begin{equation*}
\sum_{ h  \pmod{k}^{*} \atop { \sigma_1 \leq Dh^{\prime} \leq \sigma_2 }} \omega_{h,k}e^{\frac{2\pi i}{k}\left(hn + h^{\prime}m\right) } \ll  \mathrm{gcd}\left(24n+1,k \right)^{\frac{1}{2}} k^{\frac{1}{2} + \varepsilon};
\end{equation*} 
\item[(2)] We have \begin{equation*}
\frac{\sin\left(\frac{\pi a}{c}\right)}{(-1)^{ak+1}}\sum_{ h \pmod{k}^{*} \atop { \sigma_1 \leq Dh^{\prime} \leq \sigma_2 }} \frac{\omega_{h,k}}{\sin\left(\frac{\pi a h^{\prime}}{c}\right)}e^{-\frac{\pi i a^2 k_1 h^{\prime}}{c}}e^{\frac{2\pi i}{k}\left(hn + h^{\prime}m\right) } \ll \mathrm{gcd}(24n+1,k)^{\frac{1}{2}} k^{\frac{1}{2} + \varepsilon}.
\end{equation*}
\end{itemize}
The implicit constants are independent of $a$ and $k$. 
\end{lemma}
\begin{proof}
In \cite{An1} part one is proven and part two follows from part one and the proof of Lemma 3.2 in \cite{B} after defining $\tilde{c} := c$ if $k$ is odd and $\tilde{c}=2c$ if $k$ is even and checking that $e^{-\frac{\pi i a^2 k h^{\prime}}{c^2}}\sin^{-1}\left(\frac{\pi  a h^{\prime}}{c}\right)$ only depends on $h^{\prime}$ modulo $\tilde{c}$. To show this, we insert an explicit representative of the equivalence class and show that all of the terms that do not depend on $h^{\prime}$ cancel. This establishes Lemma \ref{Lem1}.
\end{proof} 
\begin{proof}[Proof of Theorem \ref{main1}]
To prove our asymptotic formula for the crank coefficients we use the Circle Method: By Cauchy's theorem we have for $n>0$
\be 
\widetilde{A}\left(\frac{a}{c};n\right) = \frac{1}{2 \pi i}\int_{C} \frac{C\left(e^{\frac{2 \pi i a}{c}};q\right)}{q^{n+1}}dq,	
\ee
where $C$ is an arbitrary path inside the unit circle surrounding $0$ exactly once counterclockwise. Choosing a circle with radius $e^{-\frac{2 \pi}{n}}$ and as a parametrisation $q = e^{-\frac{2 \pi}{n} + 2 \pi i t} $ with $0\leq t \leq 1$ gives
\be 
\widetilde{A}\left(\frac{a}{c};n\right) = \int_{0}^{1} C\left(e^{\frac{2 \pi i a}{c}};e^{-\frac{2 \pi}{n} + 2 \pi i  t}\right) e^{2\pi - 2 \pi i n t }dt.	
\ee
We define
\be 
 \vartheta^{\prime}_{h,k} := \frac{1}{k\left(k_1 +k\right)} , \qquad \vartheta^{\prime \prime}_{h,k} := \frac{1}{k\left(k_2 +k\right)},
\ee
where $\frac{h_1}{k_1}< \frac{h}{k} < \frac{h_2}{k_2}$ are adjacent Farey fractions in the Farey sequence of order $N:= \lfloor n^{1/2}\rfloor$. For more on Farey fractions see \cite{Ap}. We know that
\be 
\frac{1}{k + k_j} \leq \frac{1}{N +1} \qquad (j = 1,2).
\ee
Now we decompose the path of integration along Farey arcs $-\vartheta^{\prime}_{h,k} \leq \Phi \leq \vartheta^{\prime \prime}_{h,k}$, where $\Phi = t - \frac{h}{k}$ and $0 \leq h < k \leq N$ with $(h,k)=1$. From this decomposition of the path we can rewrite the integral along these arcs:
\be 
\widetilde{A}\left(\frac{a}{c};n \right) = \sum_{h,k} e^{-\frac{2 \pi i h n}{k}}\int_{-\vartheta^{\prime}_{h,k}}^{\vartheta^{\prime \prime}_{h,k}} C\left(e^{2\pi i \frac{a}{c}}; e^{\frac{2 \pi i}{k}(h + iz)}\right)e^{\frac{2 \pi n z}{k}}d \Phi, 
\ee
where  $z = \frac{k}{n} -k\Phi i$. We insert our transformation formula into the integral and obtain
\begin{align*}
\widetilde{A}\left(\frac{a}{c};n \right) & = i \sin\left(\frac{\pi a}{c}\right) \sum_{h,k \atop { c \mid k}} \omega_{h,k} \frac{(-1)^{ak+1}}{\sin(\frac{\pi a h^{\prime}}{c})} e^{-\frac{\pi i a^2 k h^{\prime}}{c^2} - \frac{2 \pi i h n}{k}} \\
& \times \int_{-\vartheta_{h,k}^{\prime}}^{\vartheta_{h,k}^{\prime \prime}} z^{-\frac{1}{2}}e^{\frac{2 \pi z}{k}\left(n - \frac{1}{24}\right)+ \frac{\pi}{12kz}}C\left(e^{\frac{2 \pi i a h^{\prime}}{c}};q_1\right) d \Phi \\
& -4 i \sin\left(\frac{\pi a}{c}\right) \sum_{h,k \atop {c \nmid k }} \omega_{h,k} (-1)^{ak +l} e^{-\frac{\pi i a^2 h^{\prime}  k}{c^2}+ \frac{2 \pi i h^{\prime}l a}{c^2}- \frac{2 \pi i h n}{k}} \\
& \times \int_{-\vartheta_{h,k}^{\prime}}^{\vartheta_{h,k}^{\prime \prime}} z^{-\frac{1}{2}}e^{\frac{2 \pi z}{k}\left(n - \frac{1}{24}\right)+ \frac{\pi}{12kz}}q_1^{-\frac{l^2}{2 c^2}}C\left(ah^{\prime}, l,c ; q_1\right) d \Phi \\
 & =: \Sigma_1 + \Sigma_2 .
\end{align*}
To deduce the main contribution of $\Sigma_1$ we note that the principal part of $C\left(e^{\frac{2\pi i a h^{\prime}}{c}};q_1 \right)$ in the $q_1$ variable in the limit $ z \rightarrow 0$ is $1$ and from that it is possible to write 
\be 
C\left(e^{\frac{2\pi i a h^{\prime}}{c}};q_1\right) %1 + \sum_{r \in \N}\sum_{s\pmod{k}}a(r,s) e^{\frac{2 \pi i h^{\prime}}{k}m_{r,s}}e^{-\frac{2\pi r}{kz}},
=:1 + \sum_{r \in \N}\sum_{s\pmod{c}}a(r,s) e^{\frac{2 \pi i h^{\prime}}{k}m_{r,s}}q_1^{r}
\ee 
where $m_{r,s}$ takes values in $\Z$ and $\sum_{s \pmod{c}}a(r,s)= p(r)$ for $r>1$. Only the constant term will contribute to the main term while the other terms will contribute to the error, because for large $n$ these terms are suppressed exponentially. So from that the $\Sigma_1$ part can be written as
\be 
\Sigma_1 = S_1 + S_2,
\ee
where 
\begin{align*}
 S_1 :=  i \sin\left(\frac{\pi a}{c}\right) \sum_{h,k \atop { c | k}} \omega_{h,k} \frac{(-1)^{ak+1}}{\sin(\frac{\pi a h^{\prime}}{c})} e^{-\frac{\pi i a^2 k h^{\prime}}{c^2} - \frac{2 \pi i h n}{k}}  \int_{-\vartheta_{h,k}^{\prime}}^{\vartheta_{h,k}^{\prime \prime}} z^{-\frac{1}{2}}e^{\frac{2 \pi z}{k}\left(n - \frac{1}{24}\right)+ \frac{\pi}{12kz}}d \Phi \\
\end{align*}
and 
\begin{multline*}
 S_2  : = i \sin\left(\frac{\pi a}{c}\right) \sum_{h,k \atop { c | k}} \omega_{h,k}\frac{(-1)^{ak+1}}{\sin(\frac{\pi a h^{\prime}}{c})} e^{-\frac{\pi i a^2 k h^{\prime}}{c^2}}e^{ -\frac{2 \pi i h n}{k}} \\
 \times \int_{-\vartheta_{h,k}^{\prime}}^{\vartheta_{h,k}^{\prime \prime}} z^{-\frac{1}{2}}e^{\frac{2 \pi z}{k}\left(n - \frac{1}{24}\right)+ \frac{\pi}{12kz}}\sum_{r \in \N}\sum_{s \pmod{c}}a(r,s) e^{\frac{2 \pi i h^{\prime}}{k}m_{r,s}}%e^{-\frac{2\pi r}{kz}}% 
q_{1}^{r} d \Phi. 
\end{multline*}
To bound the error term $S_2$ it is helpful to recall some easy facts:
\begin{itemize}
\item[(i)] $ z = \frac{k}{n} - i\Phi k$;
\item[(ii)]$-\vartheta_{h,k}^{\prime} \leq \Phi \leq \vartheta_{h,k}^{\prime \prime}$;
\item[(iii)]$ Re(z) = \frac{k}{n}$;
\item[(iv)] $|z|^2 = \frac{k^2}{n^2} + k^2 \Phi^2 \geq \frac{k^2}{n^2}$;
\item[(v)] $|z|^{-\frac{1}{2}} \leq k^{-\frac{1}{2}}n^{\frac{1}{2}}$;
\item[(vi)] $|z|^2 \leq \frac{k^2}{n^2} + \frac{k^2}{k^2(k+ k_2)^2} \leq \frac{2}{n}$;
\item[(vii)] $\text{Re}(z^{-1})=\frac{\text{Re}(z)}{|z|^2} \geq \frac{k}{2}$; 
\item[(viii)] $\vartheta_{h,k}^{\prime} + \vartheta_{h,k}^{\prime \prime} \leq \frac{2}{k\sqrt{n}}$.
\end{itemize}
We split the integral in the following way (this is possible because $ k_1,k_2 \leq N$):
\begin{equation} \label{intsplit}
\int_{-\vartheta_{h,k}^{\prime}}^{\vartheta_{h,k}^{\prime \prime}} = \int_{-\frac{1}{k(N+k)}}^{\frac{1}{k(N+k)}} +  \int_{\frac{1}{k(N+k)}}^{\frac{1}{k(k_2+k)}} +  \int_{-\frac{1}{k(k_1+k)}}^{-\frac{1}{k(N+k)}}.
\end{equation}
Then $S_2$ can be rewritten into three sums each sum corresponding to one of the three integrations (\ref{intsplit}):
\be 
S_2 = S_{21} + S_{22} + S_{23}.
\ee
For example 
\begin{multline*}
S_{21} = i \sin\left(\frac{\pi a}{c}\right)\sum_{h, k \atop { c \mid k } }\omega_{h,k}\frac{(-1)^{ak + 1}}{\sin(\frac{\pi a h^{\prime}}{c})} e^{-\frac{\pi i a^2 k h^{\prime}}{c^2}- \frac{2\pi i hn}{k}} \\ 
\times  \int_{-\frac{1}{k(N+k)}}^{\frac{1}{k(N+k)}} z^{-\frac{1}{2}} e^{\frac{2\pi z}{k}\left(n-\frac{1}{24}\right) + \frac{\pi}{12 k z} } \sum_{r \in \N} \sum_{s\pmod{c}}a(r,s) e^{\frac{2 \pi i h^{\prime} m_{r,s}}{k}}q_{1}^{r} d\Phi.
\end{multline*}
By taking the absolute value of this it is possible to bound the term. Before doing that we define
\be 
a(r) :=  \sum_{s\pmod{c}}|a(r,s)|.
\ee
Note that the $a(r)$ are exactly $p(r)$ except from some constant term ambiguity. We proceed:
\begin{multline*}
\left|S_{21}\right| \leq  \sum_{r=1}^{\infty} \sum_{c\mid k}\sum_{s \pmod{c}} | a(r,s)| \left|(-1)^{ak +1} \sin\left(\frac{\pi a}{c}\right)\sum_{h} \frac{\omega_{h,k}}{\sin(\frac{\pi a h^{\prime}}{c})}e^{-\frac{\pi i a^2 k h^{\prime}}{c^2} - \frac{2\pi i h n}{k} + \frac{2 \pi i m_{r,s} h^{\prime}}{k}}\right|  \\ \times  k^{-\frac{1}{2}}n^{\frac{1}{2}}  \int_{-\frac{1}{k(N+k)}}^{\frac{1}{k(N+k)}}\left| e^{\frac{2\pi z}{k}\left(n-\frac{1}{24}\right) +\frac{\pi}{12 k z} } q_{1}^{r}\right| d\Phi 
\\ \leq \sum_{r=1}^{\infty}   \sum_{c\mid k} \sum_{s\pmod{c}}|a(r,s)|\left|(-1)^{ak +1} \sin\left(\frac{\pi a}{c}\right)\right.  \left. \sum_{h} \frac{\omega_{h,k}}{\sin(\frac{\pi a h^{\prime}}{c})}e^{-\frac{\pi i a^2 k h^{\prime}}{c^2} - \frac{2\pi i h n}{k} + \frac{2 \pi i m_{r,s} h^{\prime}}{k}}\right| \\ \times k^{-\frac{1}{2}}n^{\frac{1}{2}} e^{2\pi + \frac{\pi}{12n}} e^{-\pi r} \int_{-\frac{1}{k(N+k)}}^{\frac{1}{k(N+k)}}d\Phi .
\end{multline*}
Using Lemma \ref{Lem1} (2) we may bound this up to a constant by
\begin{multline*}
\sum_{r=1}^{\infty}\sum_{k}\sum_{s \pmod{c}} |a(r,s)| e^{-\pi r}\left(24n - 1, k\right)^{\frac{1}{2}} n^{\frac{1}{2}} k^{-\frac{1}{2}} k^{\frac{1}{2} + \varepsilon}k^{-1} n^{-\frac{1}{2}} \\
=  \sum_{r=1}^{\infty} |a(r)| e^{-\pi r} \sum_{k} k^{-1 + \varepsilon}\left(24n - 1, k\right)^{\frac{1}{2}} \ll \sum_k k^{-1 + \varepsilon} \left(24n - 1, k\right)^{\frac{1}{2}} \\ \ll \sum_{k \leq N} k^{-1 + \varepsilon}\sum_{ d \mid k \atop { d \mid 24n-1}} d^{\frac{1}{2}} %\leq  C_1 \sum_{d \leq N \atop { d \mid 24n-1}} \sum_{ k/d =: k^{\prime} \leq N/d} k^{-1 + \varepsilon}d^{\frac{1}{2}}=  \tilde{C}\sum_{ d \mid 24n -1 \atop { d\leq N }} d^{\frac{1}{2}} \sum_{k^{\prime} \leq N/d }(k^{\prime} d)^{-1 + \varepsilon}
  \ll  \sum_{ d \mid 24n -1 \atop { d\leq N }} d^{\frac{1}{2}} \sum_{k\leq N/d }(k d)^{-1 + \varepsilon} \\\ll \sum_{ d \mid 24n -1 \atop { d\leq N }} d^{-\frac{1}{2}} \sum_{k\leq N/d }k^{-1} N^{\varepsilon} \ll n^{\varepsilon} \sum_{ d \mid 24n -1 \atop { d\leq N}}d^{-\frac{1}{2}} \ll n^{\varepsilon}.
\end{multline*}
We conclude that $S_{21} = O(n^{\varepsilon})$. $S_{22}$ and $S_{23}$ are bounded in the same way and so we just consider $S_{22}$. We can rewrite the integral in the following way
\begin{equation} \label{split}
\int_{-\frac{1}{k(k_1 + k)}}^{-\frac{1}{k(N + k)}} = \sum_{\ell=k_1 + k}^{N+k -1}
\int_{-\frac{1}{k\ell}}^{-\frac{1}{k(\ell+1)}}.
\end{equation}
Plugging in this splitting of the integral we obtain the bound
\begin{align*}
\left|S_{22}\right| \leq \left|  \sum_{r=0}^{\infty}  \sum_{c \mid k}\sum_{s \pmod{c}}\right. a(r,s)\sum_{\ell= \tilde{k}_1 + k}^{N+k-1}
\int_{-\frac{1}{k\ell}}^{-\frac{1}{k(\ell+1)}} z^{-\frac{1}{2}}q_{1}^{r} e^{\frac{2 \pi }{12kz} + \frac{2 \pi z}{k}\left(n-\frac{1}{24}\right)}d\Phi  \\ \left. (-1)^{ak + 1} \sin\left(\frac{\pi a }{c}\right) \sum_{h}\frac{\omega_{h,k}}{\sin\left(\frac{\pi a h^{\prime}}{c}\right)}e^{-\frac{\pi i a^2 k h^{\prime}}{c^2}}e^{-\frac{2\pi i h n}{k}}e^{\frac{2 \pi i m_{r,s} h^{\prime}}{k}} \right| =:A .
\end{align*} 
We use the condition $ N < k + k_1 \leq \ell$ and so we can rearrange the summation from $\sum_{\ell= k_1 + k}^{N+k -1}$ to $ \sum_{\ell= N + 1}^{N+k -1}$, but we also have to rewrite the sum over $h$ to count all the terms that contribute:  
\begin{align}\label{A2}
A = \left|  \sum_{r=0}^{\infty}  \sum_{c \mid k}\sum_{s \pmod{c}}\right. a(r,s)\sum_{\ell= N + 1}^{N+k -1}
\int^{-\frac{1}{k(\ell+1)}}_{-\frac{1}{k\ell}} z^{-\frac{1}{2}} q_{1}^{r} e^{\frac{2 \pi }{12kz} + \frac{2 \pi z}{k}\left(n-\frac{1}{24}\right)}d\Phi \nonumber \\ \left.(-1)^{ak + 1} \sin\left(\frac{\pi a }{c}\right) \sum_{h \atop {N < k + \tilde{k}_1 \leq  \ell}}\frac{\omega_{h,k}}{\sin(\frac{\pi a h^{\prime}}{c})}e^{-\frac{\pi i a^2 k h^{\prime}}{c^2}}e^{-\frac{2\pi i h n}{k}}e^{\frac{2 \pi i m_{r,s} h^{\prime}}{k}}  \right|.
\end{align}
Now by the theory of Farey fractions we have
\be 
k_1 \equiv - h^{\prime} \pmod{k} \, , \, k_2 \equiv  h^{\prime} \pmod{k} \, , \, N-k \leq k_i \leq N, 
\ee for $i =1,2$. This can be seen by \cite{Ap}, Theorem 5.4 where it is proven that adjacent Farey fractions fulfill some unimodular relations that are equivalent to the above statement. We see that it is possible to use Lemma \ref{Lem1}(2) to bound contributions of \eqref{A2} and with that also $S_{22}$. This is done like in the $S_{21}$ case by using the facts listed above and using the same bounds. The only difference is that we need to be careful about the bound of the sum over the different integrals. An easy calculation shows that the following bound can be obtained: 
\be 
\sum_{\ell = N +1}^{N+k-1}\int_{-\frac{1}{k \ell}}^{-\frac{1}{k \left( \ell+1\right)}} d\Phi \leq \frac{2}{k\sqrt{n}}
\ee
So all the terms can be bounded the same way. Thus, we obtain the same result:
\begin{equation*}
S_{21} =O(n^{\varepsilon}); \, S_{22} = O(n^{\varepsilon});\, S_{23} = O(n^{\varepsilon}).
\end{equation*}
So $\Sigma_1$ is equal to: 
\be
i \sin\left(\frac{\pi a}{c}\right) \sum_{h,k \atop { c \mid k}} \omega_{h,k} \frac{(-1)^{ak+1}}{\sin(\frac{\pi a h^{\prime}}{c})} e^{-\frac{\pi i a^2 k h^{\prime}}{c^2} - \frac{2 \pi i h n}{k}}  \int_{-\vartheta_{h,k}^{\prime}}^{\vartheta_{h,k}^{\prime \prime}} z^{-\frac{1}{2}}e^{\frac{2 \pi z}{k}\left(n - \frac{1}{24}\right)+ \frac{\pi}{12kz}}d \Phi  + O(n^{\varepsilon}).
\ee
Next we want to analyze $S_1$. Therefore we use a similar trick like (\ref{split}) to split the integral:
\be \label{intsplit2}
\int_{-\vartheta_{h,k}^{\prime}}^{\vartheta_{h,k}^{\prime \prime}} = \int_{-\frac{1}{kN}}^{\frac{1}{kN}}-\int_{-\frac{1}{kN}}^{-\frac{1}{k(k+ k_1)}}-\int_{\frac{1}{k(k+ k_2)}}^{\frac{1}{kN}}. 
\ee 
and denote  by $S_{11}, S_{12}, S_{13}$ the corresponding sums. It is possible to show that $S_{12}$ and $S_{13}$ contribute to the error term. We begin with $S_{12}$. Similar to the analysis of the error terms of $S_2$ we write for the integral:
\be 
\int_{-\frac{1}{kN}}^{-\frac{1}{k(k+ k_1)}}= \sum_{\ell=N}^{k + k_1 -1}\int_{-\frac{1}{k\ell}}^{-\frac{1}{k(\ell+1)}}.
\ee
Plugging into $S_{12}$ gives:
\begin{align*}
S_{12} = \sum_{c \mid k} \sum_{\ell=N}^{k + k_1 -1}\int_{-\frac{1}{k\ell}}^{-\frac{1}{k(\ell+1)}}z^{-\frac{1}{2}}e^{\frac{\pi}{12kz} + \frac{2\pi z}{k}\left(n-\frac{1}{24}\right)}d \Phi \frac{\sin\left(\frac{\pi a}{c}\right)}{(-1)^{ak+1}}\sum_{h}\frac{\omega_{h,k}}{\sin\left(\frac{\pi a h^{\prime}}{c}\right)}e^{-\frac{\pi i a^2 h^{\prime} k}{c^2} - \frac{2 \pi i h n}{k}}.
\end{align*}
Now due to the condition $k_1 \leq N$ we have that $\ell \leq k +k_1 -1 \leq N+ k-1$ which restricts the summation over $h$. We can now bound $S_{12}$ by summing over more integrals:
\begin{multline*}
|S_{12}| \leq  \sum_{c \mid k} \sum_{\ell=N}^{k +N -1}\int_{-\frac{1}{k\ell}}^{-\frac{1}{k(\ell+1)}}\left| z^{-\frac{1}{2}}e^{\frac{\pi}{12kz} + \frac{2\pi z}{k}\left(n-\frac{1}{24}\right)}\right| d \Phi \\ \times \left| \sin\left(\frac{\pi a}{c}\right)(-1)^{ak+1}\sum_{h \atop {\ell \leq k + k_1 -1 \leq N- k-1} }\frac{\omega_{h,k}}{\sin\left(\frac{\pi a h^{\prime}}{c}\right)}e^{-\frac{\pi i a^2 h^{\prime} k}{c^2}}e^{\frac{2 \pi i h n}{k}}\right| 
= O(n^{\varepsilon}),
\end{multline*}
using again Lemma \ref{Lem1} and the facts listed at the beginning of the proof.
In the next step we detect the main contributions from the second sum $\Sigma_2$ in the Circle Method. We rewrite  $\Sigma_2$ in such a way that it is easy to see if certain terms contribute to the main part using geometric series, which is possible because $|q_1|<1$: 
\begin{multline*}
C\left(a h^{\prime}, l,c ; q_1 \right) =\frac{i}{2(q_1)_{\infty}}\left(\sum_{m=0} \frac{(-1)^m e^{-\frac{\pi i a h^{\prime}}{c}}q_{1}^{\frac{m^2+m}{2} +\frac{l}{2c}}}{1-e^{-\frac{2 \pi i a h^{\prime}}{c}}q_{1}^{m + \frac{l}{c}}} \right. 
- \left. \sum_{m=1} \frac{(-1)^m e^{\frac{\pi i a h^{\prime}}{c}}q_{1}^{\frac{m^2+m}{2} -\frac{l}{2c}}}{1-e^{\frac{2 \pi i a h^{\prime}}{c}}q_{1}^{m - \frac{l}{c}}}\right) \\ 
= \frac{i}{2(q_1)_{\infty}}\left(\sum_{m=0} \right. (-1)^m \sum_{r=0}e^{-\frac{\pi i a h^{\prime}}{c} -\frac{2 \pi i r a h^{\prime}}{c}}q_{1}^{\frac{m}{2}(m+1) +\frac{l}{2c}+ rm +\frac{rl}{c}} \\ -\sum_{m=1} (-1)^m \sum_{r=0} e^{\frac{\pi i a h^{\prime}}{c} + \frac{2 \pi i a r h^{\prime}}{c}}\left. q_{1}^{\frac{m}{2}(m+1) -\frac{l}{	2c}+ rm -\frac{rl}{c}}\right).
\end{multline*}
From this expression and the following explanation we can write
\begin{equation} \label{Main1}
e^{-\frac{\pi i a^2 h^{\prime} k}{c^2}+\frac{2 \pi i h^{\prime} l a}{c^2}+ \frac{\pi}{12 k z} }q_{1}^{\frac{-l^2}{2 c^2}}C\left(ah^{\prime},l,c; q_1\right)=: \sum_{r \geq r_0}\sum_{s \pmod{c}} b(r,s) e^{\frac{2 \pi i m_{r,s} h^{\prime}}{k}} q_{1}^{r}.
\end{equation}
We next explain that $m_{r,s} \in \Z$ and $r_0$ is possibly negative. The part with negative $r$  contributes to the main part. We rewrite \eqref{Main1} further by using $1/(q_1)_{\infty} = 1 + O(q_{1})$ inside of $C(ah^{\prime},l,c; q_1)$. So, the main contribution of 
\be 
e^{-\frac{\pi i a^2 h^{\prime} k}{c^2}+\frac{2 \pi i h^{\prime} l a}{c^2}+ \frac{\pi}{12 k z} }q_{1}^{\frac{-l^2}{2 c^2}}C\left(ah^{\prime},l,c; q_1\right)
\ee
comes from the following expression:
\begin{align}\label{A}
\pm\frac{i}{2}e^{-\frac{\pi i a^2 h^{\prime} k}{c^2}+\frac{2 \pi i h^{\prime} l a}{c^2}+ \frac{\pi}{12 k z}}q_{1}^{\frac{-l^2}{2 c^2}}(-1)^m  q_{1}^{\frac{m}{2}(m+1) \pm \frac{l}{2c}+ rm \pm \frac{rl}{c}}e^{\mp \frac{\pi i a h^{\prime}}{c} \mp \frac{2 \pi i a h^{\prime}r}{c}}.
\end{align} 
From this is possible to split the expression into the roots of unity and to the part that depends on the variable $z$. The roots of unity look like
\be 
\exp\left(\frac{2 \pi i h^{\prime}}{k}\left(-\frac{a^2 k^2 }{2 c^2} +\frac{l a k}{c^2}- \frac{l^2}{2 c^2}+ rm \pm \frac{rl}{c} \mp \frac{k r a}{c} + \frac{m(m +1)}{2} \pm \frac{l}{2 c} \mp \frac{ak}{2c}\right)\right).
\ee 
Rewriting the expression in the second bracket, using  the congruence condition $ l \equiv a k \pmod{c}$, $l^2 \pm l$ is always even and rearranging the sum it is possible to show that the contribution of the roots of unity looks like $\exp\left(\frac{2 \pi i h^{\prime}m_{r,s}}{k}\right)$ where $m_{r,s}$ is a sequence in $\Z$. The interesting part happens for $ \exp\left(\frac{\pi}{kz}T\right)$, where $T$ is defined in the following way:
\be
T:= \frac{l^2}{c^2} + \frac{1}{12} - 2rm \mp 2r \frac{l}{c} - m(m+1) \mp \frac{l}{c}.
\ee
This part contributes to the circle method exactly if $T >0$ which is equivalent to $ -T < 0$. Firstly we treat the case with the plus sign in \eqref{A}. By multiplying by (-1) and assuming $m >0$ it is possible to show
\be 
-T =-\frac{l^2}{c^2}-\frac{1}{12} + 2rm + 2r \frac{l}{c} + m(m+1) + \frac{l}{c} > -1 -\frac{1}{12} + 2 +1 > 1 > 0.
\ee 
So, $-T > 0$ and this gives for all $r$ no contribution to the Circle Method. 
For $m=0$ define  $r$ to be a solution to the following inequality:
\be 
-\frac{l^2}{c^2} -\frac{1}{12} + 2 r \frac{l}{c} + \frac{l}{c} < 0.
\ee
This is equivalent to $T > 0$ and so this contributes to the main part in the Circle Method.
Now choosing the minus sign in the equation \eqref{A} that becomes
\be 
T = -\frac{l^2}{c^2} - \frac{1}{12} + 2 rm  -2 r \frac{l}{c} + m(m+1) - \frac{l}{c}.
\ee
Assuming that $m \geq 2$, it is possible to show that $ -T > 3 >0$ and this gives no contribution. For $m=1$ we define $f: [0,1] \rightarrow \R$ by
\be 
f(x):= -x^2 - x(1 + 2r) -\frac{1}{12} +2 +2r.
\ee
Calculating the maximum and computing the values of the function we see that on the boundary the function is negative, i.e., $f(1)= - \frac{1}{12} < 0$. Thus this contributes to the main part in the Circle Method. So there are two contributions coming from each of the two terms of $C(ah^{\prime},\frac{lc}{c_1},c ;q_1)$. The first one comes from the first sum, if $m=0$, and this contributes with
\be 
\frac{i}{2}e^{-\frac{\pi i a^2 h^{\prime} k}{c^2}+\frac{2 \pi i h^{\prime} l a}{c^2}- \frac{\pi i h^{\prime} a}{c}+ \frac{\pi}{12 k z} }q_{1}^{\frac{-l^2}{2 c^2} + \frac{l}{2 c}} \sum_{r \geq 0 \atop { \delta_{a,c,k,r}^{+} > 0}} e^{-\frac{2 \pi i h^{\prime} a r}{c}} q_{1}^{\frac{rl}{c}}.
\ee
%where $\delta_{a,c,k,r}^{+} = \frac{l^2}{2 c_1^2} + \frac{1}{24} -(r + 1/2) \frac{l}{c_1}$.
The second contribution comes from the second sum, if $m=1$, and this contributes with
\be 
e^{-\frac{\pi i a^2 h^{\prime} k}{c^2}+\frac{2 \pi i h^{\prime} l a}{c^2}+ \frac{\pi i h^{\prime} a}{c}+ \frac{\pi}{12 k z} }q_{1}^{\frac{-l^2}{2 c^2} + \frac{l}{2 c}+1} \sum_{r \geq 0 \atop { \delta_{a,c,k,r}^{-} > 0}} e^{\frac{2 \pi i h^{\prime} a r}{c}} q_{1}^{r\left(1-\frac{l}{c}\right)}.
\ee
%where $\delta_{a,c,k,r}^{-} = \frac{l^2}{2 c_1^2} - \frac{23}{24} -r\left(1 - \frac{l}{c_1}\right)+ \frac{l}{2c_1}$.
Thus we have for the leading order of $\Sigma_2$ the following expression:
\be 
2 \sin\left(\frac{\pi a}{c}\right)\sum_{k,r \atop {c \nmid k \atop { i \in \lbrace -,+ \rbrace }}}(-1)^{ak+l}\sum_{h}\omega_{h,k}e^{\frac{2 \pi i}{k}\left(- n h + m_{a,c,k,r}^{i} h^{\prime}\right)}\int_{-\vartheta_{h,k}^{\prime}}^{\vartheta_{h,k}^{\prime \prime}}z^{-\frac{1}{2}} e^{\frac{2 \pi z}{k}\left(n - \frac{1}{24}\right) + \frac{2 \pi}{kz}\delta_{a,c,k,r}^{i}} d \Phi.
\ee
Now it is possible to rewrite the sum over $k$ into the sum where the $k$'s have the same values for $c_1$ and $l$ and thus the  $\delta_{a,c,k,r}^{i}$ are constant in each class and the condition $\delta_{a,c,k,r}^{i}>0$ is independent of $k$ in each class. Moreover it is clear as $c$ and $l$ are finite numbers and for arbitrary large $r$ there do not exist any solutions to $\delta_{a,c,k,r}^{i}>0$, so that there are only finitely many solution to the inequality. That means it is possible to split the sum over $r$ into positive $\delta_{a,c,k,r}^{i}$, which by the above argument is a finite sum and into negative $\delta_{a,c,k,r}^{i}$, where the part with negative $\delta_{a,c,k,r}^{i}$ contributes to the error. By symmetrizing the integral and now using Lemma \ref{Lem1} (1) it is possible to bound all the terms exactly the same way we did for $\Sigma_1$:
\begin{align*}
\Sigma_2 = 2 \sin\left(\frac{\pi a}{c}\right)\sum_{k,r \atop {c \nmid k \atop  { \delta_{a,c,k,r}^{i}>0 \atop { i \in \lbrace -,+ \rbrace }}}}(-1)^{ak+l}&\sum_{h}\omega_{h,k}e^{\frac{2 \pi i}{k}\left(- n h + m_{a,c,k,r}^{i} h^{\prime}\right)} \\ & \times \int_{-\frac{1}{kN}}^{\frac{1}{kN}}z^{-\frac{1}{2}} e^{\frac{2 \pi z}{k}\left(n - \frac{1}{24}\right) + \frac{2 \pi}{kz}\delta_{a,c,k,r}^{i}} d \Phi + O(n^{\varepsilon}).
\end{align*}
Another way to argue is to plug in the expansion (\ref{Main1}) directly and split the sum over $r$ into positive and non-positive powers. Then by our analysis above we see that the coefficients of the expansion do not depend on $a$ and $k$, because the roots of unity are all expressions in $l/c$. So we can bound all the terms with $k$ by using Lemma \ref{Lem1} and as the $b(r,s)$ grow exactly like the partition function  with $r$ and so smaller than $\exp(-\frac{\pi r}{12 k z})$ the product of theses two quantities can also be bounded by a constant.
So at the end we have
\begin{equation}
\Sigma_2 = 2\sin\left(\frac{\pi a}{c}\right)\sum_{k,r \atop {c \nmid k \atop {\delta_{a,c,k,r}^{i} > 0 \atop { i\in \lbrace +,- \rbrace}}}}D_{a,c,k}(-n, m_{a,c,k,r}^{i})\int_{-\frac{1}{kN}}^{\frac{1}{kN}}z^{-\frac{1}{2}} e^{\frac{2 \pi z}{k}\left(n - \frac{1}{24}\right) + \frac{2 \pi}{kz}\delta_{a,c,k,r}^{i}} d \Phi + O(n^{\varepsilon})
\end{equation}
and by the analysis before
\begin{equation}
\Sigma_1 = i \sum_{c \mid k} \tilde{B}_{a,c,k}(-n,0)\int_{-\frac{1}{kN}}^{\frac{1}{kN}}z^{-\frac{1}{2}}e^{\frac{2\pi z}{k}\left(n -\frac{1}{24}\right) + \frac{\pi}{12kz}} d\Phi + O(n^{\varepsilon}). 
\end{equation}
To finish the proof we have to evaluate integrals of the following form: 
\be 
I_{k,t}:= \int_{-\frac{1}{kN}}^{\frac{1}{kN}} z^{-\frac{1}{2}}e^{\frac{2 \pi}{k}\left(z\left(n-\frac{1}{24}\right) +\frac{t}{z}\right)}d\Phi.
\ee
Substituting $z = k/n - i k \Phi$ gives
\be 
I_{k,t} =\frac{1}{ki}\int_{k/n- \frac{i}{N}}^{k/n+ \frac{i}{N}}z^{-\frac{1}{2}}e^{\frac{2 \pi}{k}\left(z\left(n-\frac{1}{24}\right) +\frac{t}{z}\right)}dz.
\ee
We introduce the circle through the complex conjugated points $k/n \pm i/N$ which is tangent to the imaginary axis at 0 and denote this circle by $\Gamma$. Writing a complex number on the circle by $z = x + iy$ we have as a circle equation $x^2 + y^2 = \alpha x$ with $\alpha = \frac{k}{n} + \frac{n}{N^2k}$. On the smaller arc, that is the arc going from the two complex conjugated points through zero, we clearly have $\textrm{Re}(z) \leq \frac{k}{n}$ , $\textrm{Re}\left(z^{-1}\right)< k$ and $2 > \alpha > \frac{1}{k}$. From evaluating the integral on the smaller arc we get that the integral is bounded by $O(n^{-\frac{1}{8}})$\footnote{ we will make this statement more precise in the next section, see (\ref{errint})}. So it possible to change the path of integration to the larger arc because we have no singularities enclosed by the larger arc anymore. So by Cauchy's Theorem we obtain: 
\be
I_{k,t} = \int_{\Gamma} z^{-\frac{1}{2}} e^{\frac{2 \pi}{k}\left(z\left(n-\frac{1}{24}\right) + \frac{t}{z}\right)}dz + O(n^{-\frac{1}{8}}). 
\ee
Transforming the circle to a straight line by $ s = \frac{2 \pi r}{kz}$ gives:
\be 
I_{k,t} = \frac{2 \pi}{k} \left( \frac{2 \pi t}{k}\right)^{1/2} \frac{1}{2\pi i}\int_{\gamma - i\infty}^{\gamma + i\infty} s^{-\frac{3}{2}} e^{s + \frac{\beta}{s}}ds + O\left(n^{-\frac{1}{8}}\right),
\ee 
where $\gamma \in \R$ and $\beta = \frac{\pi^2 t}{6k^2}(24n-1)$. By the Hankel integral formula \cite{B} we get
\be 
I_{k,t} = \frac{4\sqrt{3}}{\sqrt{k(24n-1)}}\sinh\left(\sqrt{\frac{2t(24n-1)}{3}}\frac{\pi}{k}\right) + O \left(n^{-\frac{1}{8}}\right).
\ee
Now at the end we have 
\begin{align*}
\Sigma_2 + \Sigma_1 & =  2\sin\left(\frac{\pi a}{c}\right)\sum_{k,r \atop {c \nmid k \atop {\delta_{a,c,k,r}^{i} > 0 \atop { i\in \lbrace +,- \rbrace}}}}D_{a,c,k}\left(-n, m_{a,c,k,r}^{i}\right)\int_{-\frac{1}{kN}}^{\frac{1}{kN}}z^{-\frac{1}{2}} e^{\frac{2 \pi z}{k}\left(n - \frac{1}{24}\right) + \frac{2 \pi}{kz}\delta_{a,c,k,r}^{i}} d \Phi \\ & + 
i \sum_{c \mid k} \widetilde{B}_{a,c,k}\left(-n,0\right)\int_{-\frac{1}{kN}}^{\frac{1}{kN}}z^{-\frac{1}{2}}e^{\frac{2\pi z}{k}\left(n -\frac{1}{24}\right) + \frac{\pi}{12kz}} d\Phi  + O(n^{\varepsilon})
\end{align*}
finishing the proof of Theorem \ref{main1} after inserting the expressions for $I_{k,t}$.
\end{proof}
Now let $M(a,c;n)$ be the number of partitions of $n$ with crank equal to $a$ modulo $c$. From the Theorem \ref{main1} it is now easy to give asymptotics for the functions $M(a,c;n)$:
\begin{proof}[Proof of Corollary \ref{corblabla}]
The corollary follows easily from the following identity
\begin{equation} \label{Cmod2}
\sum_{n=0}^{\infty} M(a,c;n)q^n = \frac{1}{c} \sum_{n=0}^{\infty}p(n)q^n + \frac{1}{c}\sum_{j=1}^{c-1} \zeta_c^{-aj} C(\zeta_c^j;q),
\end{equation}
the Rademacher formula (\ref{part}) for $p(n)$ and Theorem \ref{main1}
%Plugging in the coefficients for $C(\zeta_c^j;q)$, the Rademacher formula and comparing termwise in the $q$-expansion shows the corollary. To prove \eqref{Cmod2} we notice that for the right hand side of \eqref{Cmod2} the following holds
%\begin{align*}
%\frac{1}{c} \sum_{n=0}^{\infty}p(n)q^n + \frac{1}{c}\sum_{j=1}^{c-1} \zeta_c^{-aj} C(\zeta_c^j;q)= & \frac{1}{c}\sum_{j \pmod{c}}\sum_{n\geq 0}\sum_{m\in \Z}M(m,n) \zeta_{c}^{mj}\zeta_{c}^{-aj}q^n \\
%= & \frac{1}{c}\sum_{n\geq 0}\sum_{m\in \Z}M(m,n)\left( \sum_{j \pmod{c}}\zeta_{c}^{\left(m-a\right)j}\right)q^n,
%\end{align*}
%where the first term on the left hand side of the first equation corresponds to the case $j=0$ on the right hand side. Using the orthogonality of the roots of unity 
%$$ \sum_{j \pmod{c}} \zeta_{c}^{rj} = \begin{cases} 0 \quad r \not\equiv 0 \pmod{c}, \\ c \quad r \equiv 0 \pmod{c}, \end{cases} $$
%we obtain 
%\begin{align*} 
%\frac{1}{c} \sum_{n=0}^{\infty}p(n)q^n + \frac{1}{c}\sum_{j=1}^{c-1} \zeta_c^{-aj} C(\zeta_c^j;q) = & \sum_{n\geq 0}\sum_{m\in \Z \atop{m \equiv a \pmod{c}}}M(m,n)q^n \\ = & \sum_{n\geq 0}M(a,c;n)q^n,
%\end{align*}
%and this finishes the proof of \eqref{Cmod2}.
\end{proof}
\section{Inequalities of Crank differences}
Here we give a proof of Theorem \ref{blabla} and obtain inequalities for $c< 11$.
\begin{proof}[Proof of Theorem \ref{blabla}]
Firstly, we define
$$
\rho_j(a,b,c):=\left( 
\cos \left( \frac{2 \pi a j}{c} \right)
- \cos \left(\frac{2 \pi b j}{c} \right)
\right).
$$
It is possible to write the crank differences as (see (\ref{Cmod2}))
\begin{equation} \label{rootsum}
\sum_{n} \left( M(a,c;n)-M(b,c;n)\right)q^n
=
\frac{2}{c} \sum_{j=1}^{\frac{c-1}{2}} 
\rho_j(a,b,c)
C\left( \zeta_{c}^{j};q\right),
\end{equation}
where we define $\zeta_{c} =e^{\frac{2\pi i}{c}}$.
We deduce the asymptotic behavior of (\ref{rootsum}) using Theorem \ref{main1}. So we insert Theorem \ref{main1} into the equation (\ref{rootsum}) and get directly
\begin{equation} \label{crankdifference}
M(a,c;n) - M(b,c;n) = 
\sum_{j=1}^{\frac{c-1}{2}} 
\left(
S_j(a,b,c;n) + \sum_{ i \in \lbrace -,+\rbrace}T_j^i(a,b,c;n) +O\left(n^{\varepsilon}\right)\right)
\end{equation}
where we have 
\begin{equation} \label{Sj}
S_j(a,b,c;n):= 
 \rho_j(a,b,c)     
  \frac{8\sqrt{3} i }{ c \sqrt{24n-1}} \sum_{1 \leq k \leq \sqrt{n} \atop c|k} 
\frac{\widetilde{B}_{j,c,k}(-n,0)}{\sqrt{k}}  
\sinh \left(\frac{\pi  \sqrt{24n-1} }{6k}\right),
\end{equation}
\begin{align} \label{Tj}
T_j^i(a,b,c;n):=  
\rho_j(a,b,c) & 
 \frac{ 16 \sqrt{3}   \cdot   \sin \left(\frac{\pi j}{c} \right) }{c\sqrt{24n-1}} \nonumber \\ 
    & \times \sum_{1 \leq k\leq \sqrt{n}\atop {c \nmid k\atop {r \geq 0  \atop {\delta_{j,c,k,r}^{i}>0 }} }} 
\frac{D_{j,c,k}(-n,m_{j,c,k,r}^{i})}{\sqrt{k}}   
 \sinh \left( 
 \frac{\pi\sqrt{2 \delta_{j,c,k,r}^{i}(24n-1)}}{\sqrt{3}k}
 \right).
\end{align}
This looks similar to the rank case treated in \cite{BB1}. \\
\emph{ The main term}: \\ 
Firstly, we detect the main contribution coming from the hyperbolic sine. It is a strictly increasing function and so we have to detect the largest argument. In $S_j$ the condition $ c\mid k$ has to be fulfilled and so the largest argument occurs if $k=c$. We show using $\delta^{+}_{j,c,k,r}$ with $k=1, r=0$ and $j=1$ that the argument of the hyperbolic sine of $S_j$ is always smaller then the argument of the hyperbolic sine of $T_j^i$. As $ c \mid k$ we have to show that:
\be 
\sqrt{\frac{1}{3c^2}-\frac{1}{3c} +\frac{1}{36}}> \frac{1}{6c}.
\ee 
By squaring both sides and solving a polynomial equation we see that this is equivalent to $c >11$. So for $ c> 11$ the main contribution to the crank differences comes from $T_j^{i}$, thus we have to detect the largest argument occurring in the $T_j^{i}$. To see what is the largest argument we compare $\delta_{j,c,k,r}^{+}$ and $\delta_{j,c,k,r}^{-}$. First of all it is clear that the largest argument occurs if $r=0$ for fixed $j,k$, hence we set $r=0$ and see that $\delta_{j,c,k,0}^{-} < \delta_{j,c,k,0}^{+}$, because $ 0 < l< c$. Assuming  $\frac{l}{c} < \frac{1}{2}$, which we may do by the symmetry of the parabola in the argument $l/c$ we see that $\delta_{j,c,k,0}^{+} \leq \delta_0:= \frac{1}{2c^2} + \frac{1}{24} - \frac{1}{2c}$. For $k =1$ we get $ l=j$ and so if $j \neq 1$ we have $\delta_{j,c,1,0}^{i} < \delta_0$ if $j \neq 1$. This implies that the largest argument occurs for $k=1,\,r=0$ and $j=1$. So the main contribution is 
\be 
T_{1}^{+}(a,b,c,n)= \frac{2}{c}\rho_{1}(a,b;c)\frac{8 \sqrt{3}\sin\left(\frac{\pi}{c}\right)}{\sqrt{24n-1}}\sinh\left(\frac{\pi \sqrt{2\delta_{0} (24n-1)}}{\sqrt{3}}\right).
\ee
From this it is already possible to deduce the theorem, because for sufficiently large $n$ the main contribution comes from $T_1^{+}$. The sign of $T_1^+$ is determined by the sign of the $\rho_1$ which is positive since we have $ 0 < \frac{\pi}{c} < \frac{\pi}{11} < \frac{\pi}{2} $ which implies that in this range the $\cos(\pi x/c)$ is decreasing. Thus for $ 0 < a<b < \frac{c-1}{2}$ we have that $ \cos\left(\frac{\pi a}{c}\right) > \cos \left(\frac{\pi b}{c}\right)$ and that explains why $\rho_1$ is positive. So for sufficiently large $n$ we have $N(a,c;n) > N(b,c;n)$. The next step is to clarify what sufficiently large exactly means by bounding all the error terms explicitly in terms of $c$ and $n$, beginning with the contributions of $S_j$, $T_{j}^{i}$ for $j > 1$ and $ T^{-}_{1}$. \\
\emph{ Bounding the contributions of $S_j$:}\\
For $S_j$ it is easily seen that
\begin{align*}
|S_j(a,b,c)|  \leq &  \frac{8 |\rho_j(a,b,c)|\sqrt{3}}{c \sqrt{24n-1}} \sum_{1 \leq k \leq \sqrt{n} \atop { c\mid k }} \frac{|\tilde{B}_{j,c,k}(-n,0)|}{\sqrt{k}}\sinh\left(\frac{\pi \sqrt{24n-1}}{6k}\right) \\
\leq & 
\frac{8 |\rho_j(a,b,c)|\sqrt{3}}{c \sqrt{24n-1}}\left|\sin\left(\frac{\pi j}{c}\right)\right|\sinh\left(\frac{\pi \sqrt{24n-1}}{6c}\right)\sum_{1 \leq k \leq \sqrt{n} \atop { c\mid k }}\frac{1}{\sqrt{k}}\sum_{h=1 \atop { (h,k) =1 }}^{k} \frac{1}{|\sin(\frac{\pi h}{c})|}.
\end{align*}
Here we used that the largest argument in the hyperbolic sine occurs if $c=k$ and that $h$ and $h^{\prime}$ run over the same primitive residue classes modulo $k$ and so we changed in the summation the argument of the sine from $j h^{\prime} \rightarrow h$ and with that to another representative of the equivalence class. Here it is important to note that we are using that $c$ is prime. The inner sum can be further estimated by
\begin{equation}\label{sin}
\sum_{h=1 \atop { (h,k) =1 }}^{k} \frac{1}{|\sin(\frac{\pi h}{c})|} \leq \frac{2k}{c}\sum_{h=1}^{\frac{c-1}{2}}\frac{1}{|\sin(\frac{\pi h}{c})|}\leq  \frac{2k}{\pi}\sum_{h=1}^{\frac{c-1}{2}}\frac{1}{h\left(1 - \pi^2/24\right)}\leq \frac{2k\left(1 + \log\left(\frac{c-1}{2}\right)\right)}{\left(1 - \pi^2/24\right)} .
\end{equation}
In the first inequality it used that the absolute value of the sine is not bigger than $1$ and that $c$ is odd. In the second inequality it is used that $\sin(x) > x -x^3/6$ for $ |x|<1$ and we used that the summation runs to $(c-1)/2$ by bounding in the $x^3$-term $h$ by $c/2$. In the last step we have used $\sum_{h=1}^{\frac{c-1}{2}}h^{-1} = 1 + \sum_{h=2}^{\frac{c-1}{2}} h^{-1}$ and estimated the sum by an integral. We now have:
\begin{align*}
|S_j(a,b,c)|  \leq & \frac{16 |\rho_j(a,b,c)|\sqrt{3}}{c \sqrt{24n-1}}\frac{\left|\sin\left(\frac{\pi j}{c}\right)\right|\left(1+ \log\left(\frac{c-1}{2}\right)\right)}{\pi\left(1 -\frac{\pi^2}{24}\right)}\sinh\left(\frac{\pi \sqrt{24n-1}}{6c}\right)\sum_{1 \leq k \leq \sqrt{n} \atop { c\mid k }}k^{1/2} \\
\leq & \frac{64 n^{3/4} \left(1+ \log\left(\frac{c-1}{2}\right)\right)}{\sqrt{24n-1}c^2\sqrt{3}\pi \left(1 -\frac{\pi^2}{24}\right)} \sinh\left(\frac{\pi \sqrt{24n-1}}{6c}\right).
\end{align*}
Here it is used that $\left| \rho_j(a,b,c)\right|\leq 2$ and the following estimation of the sum:  
\be \sum_{ 1 \leq k \leq \sqrt{n} \atop {c\mid k}}k^{\frac{1}{2}} \leq c^{\frac{1}{2}}\sum_{1 \leq j \leq \lfloor \frac{N}{c}\rfloor} j^{\frac{1}{2}} \leq c^{\frac{1}{2}} \int_1^{\lfloor \frac{N}{c}\rfloor} x^{\frac{1}{2}} dx \leq \frac{2}{3c}n^{\frac{3}{4}}.
\ee
Next we want to bound the $T_j^i$ for $ j\geq 2$ and $T_1^{-}$. Firstly notice that is possible to bound $D_{j,c,k}(-n, m_{j,c,k,r}^{i})$ trivially by $k$. The reason is that we sum over roots of unity and the sum runs over all primitive residue classes modulo $k$. Moreover we can bound the hyperbolic sine by the positive part because $\sinh(x) =(e^{x}-e^{-x})/2$. \\
\emph{ Bounding the contributions of $T_{j}^{i}$ for $j \geq 2$}: \\
Using the exponential function we can bound the terms in the sum of $T_j^i$ in the following way (here for $k \geq 2$):
\be 
\frac{D_{j,c,k}(-n,m_{j,c,k,r}^{i})}{\sqrt{k}}\sinh\left(\frac{\pi \sqrt{2 \delta_{j,c,k,r}^i (24n-1)}}{\sqrt{3}k}\right) \leq \frac{k^{\frac{1}{2}}}{2}e^{\frac{\pi \sqrt{2 \delta_{0} (24n-1)}}{2\sqrt{3}}}.
\ee
The number of $r$ satisfying the condition $\delta_{j,c,k,r}^{i}>0$ can be bounded in terms of $c$: First of all we find the number of solutions to the equation as a function of $l$ for fixed $c$. Now define the function $g_{c}: [1,c-1] \rightarrow \R$ by $f(l)= \frac{l}{2c} + \frac{1}{2} + \frac{c}{24l}$. We added one to the equation to afterwards take the floor function. The largest values occur on the boundary of the interval, namely $l=1$ and $l=c-1$, as the function has its minimum in the interior of the interval and is a continuous function. For $c> 11$ the function take its maximum for $l=1$ and may be bounded by $\frac{(c+18)}{24}$. For the other cases we checked by hand that the number of solutions to the equation $\delta >0$, which is $\lfloor \frac{l}{2c} + \frac{1}{2} + \frac{c}{24l} \rfloor$, can be bound by $\frac{(c+18)}{24}$, where we inserted the maximizing $l=c-1$. Thus we can bound $T_j^i$ for $k \geq 2$ by 
\be
\frac{4(c+18)}{3\sqrt{3}c \sqrt{24n-1}}n^{3/4}e^{\pi \frac{\sqrt{2 \delta_0(24n-1)}}{2 \sqrt{3}}}.
\ee
Since $\delta_{j,c,,1,0} < \delta_0$ is decreasing in $j$, for $j>1$ we bound the $k=1$ contribution by the argument of $j=2$
\be 
\frac{2(c+18)}{\sqrt{3}c \sqrt{24n-1}}e^{\pi \frac{\sqrt{2 \delta_{2,c,1,0}(24n-1)}}{\sqrt{3}}}.
\ee  
Before coming to the error terms of the Circle Method we have to bound the contribution of $T_1^{-}$.\\
\emph{ Contribution of $T_{1}^{-}$}: \\
By the same analysis, it is possible to bound this term by a similar expression as the ones before. By bounding the $\sinh$ by the exponential function, using $\rho\leq 2$  and showing $\delta_{1,c,k,r}^{-}$ is smaller then $\delta_{2,c,1,0}^{+}$ and so choosing the right argument in the exponential function, $T_{1}^{-}$ can be bounded by
\be 
\frac{2(c-1)}{\sqrt{3}c \sqrt{24n-1}}e^{\pi \frac{\sqrt{2 \delta_{2,c,1,0}^{+}(24n-1)}}{\sqrt{3}}}.
\ee
Here we bounded the number of solutions to $\delta_{j,c,k,r}^{-}>0$ by $\frac{c-1}{24}$ which is a rough bound, but makes sense for all odd $c$ (We could find a sharper bound for the number of solutions, but we would have had to put an extra condition on $c$ or introduce a heavyside function that reflects the fact that there are no solutions for $c < 23$). Now we want to make the $O(n^{\varepsilon})$-term in the Theorem \ref{main1} explicit. We had $\tilde{A}\left(\frac{j}{c};n\right)= \Sigma_1 + \Sigma_2$ with $\Sigma_1 = S_1 +S_2$, where $S_2 =: S_{err}$ contributes to the error in the circle method.  \\
\emph{ Contribution of the error of $\Sigma_1$}:\\
We again use the bounds $|z| \geq \frac{k}{n} $, $\textrm{Re}(z)= \frac{k}{n}$, and $C(\zeta_{c}^{h},q_1)=1 +C(\zeta_{c}^{h},q_1)-1$ to obtain 
\be 
S_{err} \leq 2 \left|\sin\left(\frac{\pi j}{c}\right)\right|e^{2\pi}\sum_{k\leq N \atop { c \mid k}}k^{-\frac{3}{2}}\sum_{h=1 \atop { (h,k)=1}}^{k-1}\frac{1}{\sin\left(\frac{\pi h}{c}\right)}\max_{z}\left|e^{\frac{\pi}{12kz}}\left( C(\zeta_{c}^{h},q_1) - 1\right)\right|.
\ee 
Here we used again a change of variables $ j h^{\prime} \rightarrow h$.
The next step is to estimate $|e^{\frac{\pi}{12kz}}(C(\zeta_{c}^{h},q_1)-1)|$. Remember that \eqref{pochh}:
\begin{align*}
C(\zeta_{c}^{h},q_1)& = \frac{1}{(q_1)_{\infty}}+\frac{(1-\zeta_{c}^{h})}{(q_1)_{\infty}}\sum_{m\in \Z \setminus \lbrace 0 \rbrace} \frac{(-1)^m q_1^{\frac{m(m+1)}{2}}}{1 - \zeta_c^h q_1^m} \\
&=  \frac{1}{(q_1)_{\infty}}+\frac{(1-\zeta_c^h)}{(q_1)_{\infty}}\sum_{m =1}^{\infty} \frac{(-1)^m q_1^{\frac{m(m+1)}{2}}}{1 - \zeta_c^h q_1^m} +\frac{(1-\zeta_c^{-h})}{(q_1)_{\infty}}\sum_{m =1}^{\infty} \frac{(-1)^m q_1^{\frac{m(m+1)}{2}}}{1 - \zeta_c^{-h} q_1^m}.
\end{align*}
From this it is easily seen that $|e^{\frac{\pi}{12kz}}(C(\zeta_{c}^{h},q_1)-1)|$ may be bounded by
\be 
e^{\frac{\pi}{24}}\sum_{n=1}^{\infty} p(n)e^{-\pi n} + 
e^{\frac{\pi}{24}}\sum_{n=0}^{\infty} p(n)e^{-\pi n}\sum_{m=1}^{\infty}e^{-\pi m(m+1)/2}\left|\frac{1-\zeta_{c}^{h}}{1-\zeta_{c}^{h}q_1^{m}} + \frac{1-\zeta_{c}^{-h}}{1-\zeta_{c}^{-h}q_1^{m}}\right| .
\ee
Note that the summation index $n$ on the first term starts with $1$ where one the second sum with $0$, that means we absorbed the $-1$ into the sum. We bound the term further by noting that
\be 
\left|\frac{1-\zeta_{c}^{h}}{1-\zeta_{c}^{h}q_1^{m}} + \frac{1-\zeta_{c}^{-h}}{1-\zeta_{c}^{-h}q_1^{m}}\right| \leq 2 \frac{1 +|\cos \left(\frac{\pi}{c}\right)|}{1 - e^{-\pi m}}.
\ee
Defining $c_2 := \sum_{n=1}^{\infty}p(n) e^{-\pi n}$ and $c_1 := \sum_{m=1}^{\infty} \frac{e^{-\frac{\pi m(m+1)}{2}}}{1-e^{-\pi m}}$ it is possible to bound $S_{err}$ by:
\be 
2 e^{2\pi} \left|\sin\left(\frac{\pi j}{c}\right)\right|e^{\frac{\pi}{24}}\left(c_2 + 2\left(1 +\left|\cos\left(\frac{\pi}{c}\right)\right|\right) c_1(1+c_2)\right) \sum_{k\leq N \atop { c \mid k}}k^{-\frac{3}{2}}\sum_{h=1 \atop { (h,k)=1}}^{k-1}\frac{1}{|\sin\left(\frac{\pi h}{c}\right)|}.
\ee
Using (\ref{sin}) and estimating the sum over $k$ by an integral expression  we obtain after evaluating the integral the upper bound for $S_{err}$:
\be 
\frac{2 e^{2\pi + \frac{\pi}{24}} \left|\sin\left(\frac{\pi j}{c}\right)\right|\left(c_2 + 2\left(1 +|\cos\left(\frac{\pi}{c}\right)|\right) c_1(1+c_2)\right)n^{\frac{1}{4}}\left( 1 + \log\left(\frac{c-1}{2}\right)\right)}{\pi\left(1 -\frac{\pi^2}{24}\right)c} .
\ee
We continue by repeating this procedure for the $O(n^{\varepsilon})$-term of $\Sigma_2$.\\
\emph{Contribution of the error of $\Sigma_2$}:\\ 
The error corresponds to the terms where $\delta_{j,c,k,r}^{i}$ is not positive. Therefore we define $\tilde{M}(jh^{\prime}, l,c;q_1)$ to be the terms with 
positive exponents in the $q_1$-expansion of $$e^{\frac{\pi}{12kz}}q_1^{-\frac{l^2}{2c^2}}C(jh^{\prime},l,c;q_1)$$ and bound $\tilde{M}$. Writing for the entire sum $T_{err}$, using the usual bounds of $|z|$ and doing the usual change of variable $ jh^{\prime} \rightarrow h$ the following bound is easily obtained:
\be 
T_{err} \leq 8e^{2\pi}\left|\sin\left(\frac{\pi j}{c}\right)\right|\sum_{h,k \atop { c \nmid k}}k^{-\frac{3}{2}}\max_{z}\tilde{M}(h,l,c;q_1).
\ee  
The difficult bounds come from the function $\tilde{M}$. Remember that $C(h,l,c;q_1)$ has the following expansion
\begin{align*}
C(h,l,c,q_1) = \frac{i \zeta_{2c}^{-h} q_1^{\frac{l}{2c}}}{2 (q_1)_{\infty}\left(1 - \zeta_c^{-h} q_1^{\frac{l}{c}}\right)} + \frac{i \zeta_{2c}^h q_1^{-\frac{l}{2c}+ 1}}{2 (q_1)_{\infty}\left(1 - \zeta_c^h q_1^{1 -\frac{l}{c}}\right)} \\
-\frac{i \zeta_{2c}^h q_1^{-\frac{l}{2c}}}{2 (q_1)_{\infty}} \sum_{m=2}^{\infty} \frac{(-1)^m q_1^{\frac{m(m+1)}{2}}}{\left(1 -\zeta_c^h q_1^{m-\frac{l}{c}}\right)} + \frac{i \zeta_{2c}^{-h} q_1^{\frac{l}{2c}}}{2 (q_1)_{\infty}} \sum_{m=1}^{\infty} \frac{(-1)^m q_1^{\frac{m(m+1)}{2}}}{\left(1 - \zeta_c^{-h} q_1^{m-\frac{l}{c}}\right)}.
\end{align*}
We bound the contributions of $\tilde{M}$ term by term beginning with the first one. We rewrite the denominator by a geometric series and by $ \frac{1}{(q_1)_{\infty}} = \sum_{m=0}^{\infty} p(m)q_1^m$.
So we have
\be 
\frac{i \zeta_{2c}^{-h}q_1^{\frac{l}{2c}}}{2 (q_1)_{\infty}\left(1 -\zeta_{c}^{-h} q_1^{\frac{l}{c}}\right)}= \frac{i}{2}\zeta_{2c}^{-h}q_1^{\frac{l}{2c}} \sum_{m=0}^{\infty} p(m)q_1^m \sum_{r =0}^{\infty} \zeta_{c}^{-hr} q_1^{\frac{rl}{c}}.
\ee
Maximizing $|z|$ and taking the absolute value of this expression, noting that for $m=0$ not all the terms correspond to the error but for all higher $m$ they do, and using that Re$(z^{-1})\geq \frac{k}{2}$, we gain the following contribution to $\tilde{M}$
\be 
\frac{1}{2}e^{-\frac{\pi l}{2c}+ \frac{l^2}{2 c^2} + \frac{\pi}{24}}\left( \sum_{r \geq r_0} e^{-\frac{\pi l r}{c}} + \sum_{r=0}^{\infty} e^{-\frac{\pi l r}{c}}\sum_{m=1}^{\infty} p(m)e^{-\pi m} \right), 
\ee
where $r_0 := \lceil -\frac{1}{2} + \frac{l}{2c} + \frac{c}{24l}\rceil $. Using 
\be 
\sum_{r \geq r_0} e^{-\frac{\pi l r}{c}} = \frac{e^{-\frac{\pi r_0 l}{c}}}{\left( 1 -e^{\frac{-\pi l}{c}} \right)} \quad , \quad c_2= \sum_{m=1}^{\infty}p(m)e^{-\pi m}
\ee
and the usual geometric series we can bound the term further by 
\begin{align*}
\frac{e^{-\frac{\pi l}{2c}+ \frac{\pi l^2}{2 c^2} + \frac{\pi}{24}-\frac{\pi r_0 l}{c}}}{2\left( 1 -e^{\frac{-\pi l}{c}} \right)} + 
\frac{e^{-\frac{\pi l}{2c}+ \frac{l^2}{2 c^2} + \frac{\pi}{24}}c_2}{2\left( 1 -e^{\frac{-\pi l}{c}} \right)} & = \frac{e^{\frac{\pi l}{c}\left(-\frac{1}{2} +\frac{l}{2c} + \frac{c}{24l}-r_0 \right)}\left(1 + c_2 e^{\frac{\pi r_0 l}{c}}\right)}{2\left( 1 -e^{\frac{-\pi l}{c}} \right)} \\
& \leq \frac{\left( 1 + c_2 e^{\pi \delta_0} \right)}{2 \left( 1 - e^{-\frac{\pi}{c}}\right)}.
\end{align*}
The second sum can be bounded exactly the same way. In the third and fourth summand all the terms will contribute to the error as  was shown in the Theorem \ref{main1}. We obtain
\begin{align*}
\left| -\frac{i \zeta_{2c}^h q_1^{-\frac{l}{2c} + \frac{l^2}{2 c^2}} e^{\frac{\pi}{12 kz}}}{2 (q_1)_{\infty}} \sum_{m=2}^{\infty} \frac{(-1)^m q_1^{\frac{m(m+1)}{2}}}{\left(1 -\zeta_c^h q_1^{m-\frac{l}{c}}\right)} \right| \leq  & \frac{1}{2} e^{-\frac{\pi l}{2c} + \frac{\pi}{24} + \frac{\pi l^2}{2 c^2}} (c_2 +1 ) \sum_{m=2}^{\infty} \frac{e^{-\frac{\pi m(m+1)}{2}}}{1 - e^{-\pi m + \pi \frac{l}{c}}} \\ \leq & \frac{1}{2} e^{\pi \delta_0}(c_2 +1 ) c_3,
\end{align*}
where 
\be 
c_3 := \sum_{m=2}^{\infty} \frac{e^{-\frac{\pi m(m+1)}{2}}}{1 - e^{-\pi m + \pi }}.
\ee
We proceed by repeating this step for the fourth sum in the $C(h,l,c;q_1)$ expansion. As there is nothing new we omit this step and just give the bound for $\tilde{M}$. It is 
\be 
e^{\pi \delta_0} c_1( 1 +c_2).
\ee
So at the end the error terms coming from the function $\tilde{M}$ can be bounded by the following function $f(c)$ that just depends on $c$
\be 
f(c) := \frac{1 + c_2 e^{\pi \delta_0}}{\left( 1 - e^{-\frac{\pi}{c}}\right)} + e^{\pi \delta_0} c_1( 1 +c_2) +\frac{1}{2} e^{\pi \delta_0 }(c_2 +1 ) c_3.
\ee
So the error can be bounded by
\be 
T_{err} \leq 8 e^{2\pi} f(c) \sum_{h,k \atop { h \nmid k}} k^{-\frac{3}{2}} \leq 16 e^{2\pi} f(c) n^{\frac{1}{4}} \left| \sin\left(\frac{\pi j }{c}\right) \right|. 
\ee
Here it used that the sum over $h$ can be bounded trivially by $k$ because the sum runs over the residue class modulo $k$ and we estimated the sum over $k$ by an integral expression. 
As a next step we want to bound the contributions that come from symmetrizing the integral. In Theorem \ref{main1} we have shown that it is possible to split the integration over the Farey arcs by making the integral bounds symmetric and showed that the needed integral to correct this symmetrization will contribute to the error. These terms have to be made explicit. \newline 
\emph{ Symmetrizing}: \\
We have used 
\be 
\int_{-\vartheta_{h,k}^{\prime}}^{\vartheta_{h,k}^{\prime \prime}} = \int_{-\frac{1}{kN}}^{-\frac{1}{kN}} - \int_{-\frac{1}{kN}}^{-\frac{1}{k(k+k_1)}} - \int_{\frac{1}{k(k+k_1)}}^{\frac{1}{kN}}.
\ee
Plugging into the first term of the main contribution we are left with the following error term:
\begin{align*}
S_{1 err} = - i \sin\left(\frac{\pi j}{c}\right) \sum_{h,k \atop { c \mid k}} \frac{\omega_{h,k} (-1)^{ak+1}}{\sin\left(\frac{\pi j h^{\prime}}{c}\right)} e^{-\frac{\pi i j^2 k h^{\prime}}{c^2} -\frac{2 \pi i h n}{k}} \\ 
\times \left( \int_{-\frac{1}{kN}}^{-\frac{1}{k(k+k_1)}} + \int_{\frac{1}{k(k+k_1)}}^{\frac{1}{kN}} \right) z^{-\frac{1}{2}} e^{\frac{2 \pi z}{k}\left(n -\frac{1}{24}\right) + \frac{\pi}{12 k z}} d\Phi. 
\end{align*} 
Taking the absolute value and the usual bound of $|z|$ we can estimate
\begin{align*} 
|S_{1 err}|\leq & \left| \sin\left(\frac{\pi j}{c}\right)\right|e^{2\pi + \frac{\pi}{12}}\sum_{1 \leq k \leq \sqrt{n} \atop {c \mid k}} \frac{n^{\frac{1}{2}}}{k^{\frac{1}{2}}} \sum_{ h \atop {(h,k)=1}}\frac{1}{|\sin\left(\frac{\pi h}{c}\right)|}\frac{2}{kN} \\ \leq & \frac{4\left| \sin\left(\frac{\pi j}{c}\right)\right|e^{2\pi + \frac{\pi}{12}}\left( 1 + \log\left(\frac{c-1}{2}\right)\right)}{\pi \left(1 - \frac{\pi^2}{24}\right)}\sum_{1 \leq k \leq \sqrt{n} \atop {c \mid k }} k^{-\frac{1}{2}} \\ \leq & \frac{8e^{2\pi + \pi/12}\left( 1 + \log\left(\frac{c-1}{2}\right)\right)n^{\frac{1}{4}}}{\pi \left(1 - \frac{\pi^2}{24}\right)c} .
\end{align*}
Now we do the same for the second main contribution. Remember that we have to bound the following term:
\begin{align*}
S_{2 err} = 2 \sin\left(\frac{\pi a}{c}\right)\sum_{k,r \atop {c \nmid k \atop  { \delta_{a,c,k,r}^{i}>0 \atop { i \in \lbrace -,+ \rbrace }}}}(-1)^{ak+l}&\sum_{h}\omega_{h,k}e^{\frac{2 \pi i}{k}(- n h + m_{a,c,k,r}^{i} h^{\prime})} \\ & \times \left( \int_{-\frac{1}{kN}}^{-\frac{1}{k(k+k_1)}} + \int_{\frac{1}{k(k+k_1)}}^{\frac{1}{kN}}\right)z^{-\frac{1}{2}} e^{\frac{2 \pi z}{k}\left(n - \frac{1}{24}\right) + \frac{2 \pi}{kz}\delta_{a,c,k,r}^{i}} d \Phi
\end{align*} 
Completely analogously to $S_{1 err}$ it is possible to show:
\be 
S_{2 err} \leq 8 e^{2 \pi} \left| \sin\left(\frac{\pi j}{c}\right)\right|\sum_{r,k \atop { \delta_{j,c,k,r}^{i} >0 \atop { i\in \lbrace -,+\rbrace}}} k^{-\frac{1}{2}} e^{2\pi \delta_{j,c,k,r}^{i}}.
\ee 
As a next step we evaluate the sum over $r$ with $i=+$ and bound it in terms of $\delta_0$. As it is the biggest argument, we can also bound the term with $i=-$ analogously. So we restrict to the case $i=+$. The sum over $r$ gives
\begin{align*} 
S_{2 err} \leq & 16 e^{2 \pi} \left| \sin\left(\frac{\pi j}{c}\right)\right|\sum_{r,k \atop { \delta_{j,c,k,r}^{+} >0 \atop { i\in \lbrace -,+\rbrace}}} k^{-\frac{1}{2}} e^{2\pi \delta_{j,c,k,r}^{+}} \\
= & 16 e^{2 \pi} \left| \sin\left(\frac{\pi j}{c}\right)\right|\sum_{k} k^{-\frac{1}{2}} \sum_{r \leq r_0 -1} e^{- \frac{\pi l}{c} + \frac{\pi l^2}{c^2} + \frac{\pi}{12} - \frac{2 \pi l r}{c}}\\ = & 16 e^{2 \pi} \left| \sin\left(\frac{\pi j}{c}\right)\right|\sum_{k} k^{-\frac{1}{2}} \frac{e^{- \frac{\pi l}{c} + \frac{\pi l^2}{c^2} + \frac{\pi}{12} }\left( e^{-\frac{2 \pi l}{c}r_0}-1\right)}{e^{-\frac{2\pi l}{c}}-1} \\
\leq & 16 e^{2 \pi} \left| \sin\left(\frac{\pi j}{c}\right)\right|\sum_{k} k^{-\frac{1}{2}} \frac{e^{2\pi \delta_0}}{1 - e^{-\frac{2\pi}{c}}}\\
\leq & 32e^{2 \pi} n^{\frac{1}{4}} \left| \sin\left(\frac{\pi j}{c}\right)\right| \frac{e^{2\pi \delta_0} }{1 - e^{-\frac{2\pi}{c}}}.
\end{align*}
Two facts should be explained. We have used that the summation over $r$ is an error term if it starts with $r_0$, so here we have to sum over all the $r$-terms where $r \leq r_0 -1 $. As a next step we calculated the geometric series in $r$ and bounded the term similar to the $\Sigma_2$-error. Finally an estimation of the $k$-sum gives the final expression. The last contribution we have to bound is coming from the evaluation of the integral. There it is used that it is possible to change the path of integration if one accounts the integral over the smaller arc. This term has to be made explicit.\\
\emph{Integrating along the smaller arc}: \\
Remember that we had to compute integrals of the following form
\be 
I_{k,t} = \frac{1}{ki}\int_{-\frac{1}{kN}}^{\frac{1}{kN}}z^{-\frac{1}{2}} e^{\frac{2\pi}{k}\left( z\left(n-\frac{1}{24}\right) +\frac{t}{z}\right)}dz.
\ee
Now we denote the circle through $\frac{k}{n}\pm \frac{i}{N}$ and tangent to the imaginary axis at $0$ by $\Gamma$. For $ z = x + iy$, $\Gamma$ is given by $x^2 + y^2 =\frac{k}{n} +\frac{n}{N^2 k} x =: \alpha x$. The path of integration can be changed into the larger arc, while on the smaller arc we have the following bounds: $\frac{1}{k} < \alpha <2$, Re$(z) \leq \frac{k}{n}$ and Re$(z^{-1}) < k$. This can be used to bound the integral over the smaller arc which we denoted by $\Gamma_{S}$. Splitting $I_{k,t} = I_{k,t}^{main} + I_{k,t}^{err}$ we can bound $I_{k,t}^{err}$
\begin{align*}
I_{k,t}^{err} \leq & \frac{2}{k}e^{2 \pi + 2\pi t}\int_{\Gamma_{S}}|z|^{-\frac{1}{2}} dz \leq \frac{2}{k}e^{2 \pi +\frac{\pi}{12}}\left|\int_{0}^{\frac{k}{n}}(x^2 +y^2)^{-\frac{1}{4}}(dx + idy)\right| \\
= & \frac{2}{k}e^{2 \pi + 2\pi t}\left|\int_{0}^{\frac{k}{n}}(\alpha x)^{-\frac{1}{4}}(dx + idy)\right| = \frac{2}{k}e^{2 \pi + 2\pi t}\left|\int_{0}^{\frac{k}{n}}(\alpha x)^{-\frac{1}{4}}dx + i \int_{0}^{\frac{k}{n}}(\alpha x)^{-\frac{1}{4}}dy\right| \\
= & \frac{2}{k}e^{2 \pi + 2 \pi t} \alpha^{-\frac{1}{4}}\left|\int_{0}^{\frac{k}{n}} x^{-\frac{1}{4}}dx + i \int_{0}^{\frac{k}{n}} x^{-\frac{1}{4}}dy\right| \\ = & \frac{2}{k}e^{2 \pi + 2 \pi t} \alpha^{-\frac{1}{4}}\left|\frac{4}{3}\left(\frac{k}{n}\right)^{\frac{3}{4}} + i \int_{0}^{\frac{k}{n}} x^{-\frac{1}{4}} \dfrac{dy}{dx}dx\right|  \\ 
= & \frac{2}{k}e^{2 \pi + 2 \pi t} \alpha^{-\frac{1}{4}}\left|\frac{4}{3}\left(\frac{k}{n}\right)^{\frac{3}{4}} + i \int_{0}^{\frac{k}{n}} x^{-\frac{3}{4}} \frac{\alpha - 2x}{2 \sqrt{\alpha-x}}dx\right|.
\end{align*} 
Next we compute the derivative of $f: [0,\alpha] \rightarrow \R$ defined by $ x \rightarrow \frac{\alpha -2x}{2 \sqrt{\alpha -x}}$ and see that the derivative is negative for $x < \frac{3}{2}\alpha$ and so for all $x \in [0,\alpha]$. That means that the function $f(x)$ has its maximum at $ x=0$ and so we can bound further:
\begin{align}\label{errint}
I_{k,t}^{err}\leq & \frac{2}{k}e^{2 \pi + 2\pi t} \alpha^{-\frac{1}{4}}\left|\frac{4}{3}\left(\frac{k}{n}\right)^{\frac{3}{4}} + i\alpha^{\frac{1}{2}}\int_{0}^{\frac{k}{n}}x^{-\frac{3}{4}}dx \right| \nonumber \\ \leq &
\frac{2}{k}e^{2 \pi +2 \pi t} \left( \frac{4}{3}\left(\frac{k}{n}\right)^{\frac{3}{4}} \alpha^{-\frac{1}{4}} + 2 \alpha^{\frac{1}{4}}\left(\frac{k}{n}\right)^{\frac{1}{4}} \right) \nonumber \\ \leq & \frac{2}{k}e^{2 \pi + 2\pi t}\left( \frac{4}{3} + 2^{\frac{5}{4}}\right) n^{-\frac{1}{8}} .
\end{align}
Here it is used that $k \leq \sqrt{n}$ and so $\frac{k}{n} \leq n^{-1/2}$, $\alpha <2$. Combining the contributions from $\Sigma_1$, using usual formulas like (\ref{sin}),  and estimation of the sum over $k$, we can bound the whole contribution by ($t=1/24$):
\be 
\frac{4\left(\frac{4}{3} +2^{\frac{5}{4}}\right) \left|\sin\left(\frac{ \pi j}{c} \right)\right|\left(1 + \log\left(\frac{c-1}{2}\right)\right)e^{2\pi + \frac{\pi}{12}}n^{\frac{3}{8}}}{\pi c \left(1 -\frac{\pi^2}{24}\right)}.
\ee
The same can be done for $\Sigma_2$,
\begin{align*} 
& \, 4 \left( \frac{4}{3} + 2^{\frac{5}{4}}\right)\left|\sin\left(\frac{\pi j}{c} \right)\right| \frac{e^{2\pi \delta_0 +2\pi}}{1-e^{-\frac{2\pi}{c}}}n^{-\frac{1}{8}}\sum_{k\leq \sqrt{n}}\frac{1}{k} \\ \leq & \, 4 \left( \frac{4}{3} + 2^{\frac{5}{4}}\right)\left|\sin\left(\frac{\pi j}{c} \right)\right|\frac{e^{2\pi \delta_0 +2\pi}}{1-e^{-\frac{2\pi}{c}}}\left( 1 + \log(\sqrt{n})\right)n^{-\frac{1}{8}} \\ \leq & \, 8 \left( \frac{4}{3} + 2^{\frac{5}{4}}\right)\left|\sin\left(\frac{\pi j}{c} \right)\right|\frac{e^{2\pi \delta_0 +2\pi}}{1-e^{-\frac{2\pi}{c}}} ,
\end{align*}
where it is used that the $\log(x)$ grows more slowly than any positive power of $x$. That means that we can bound the contribution of $n$. Explicitly we calculated the maximum of $f(x)= (1+\log(\sqrt{x}))x^{-\frac{1}{8}}$ and the maximal value can be bounded by $2$.\\ 
\emph{The explicit constant}:\\
Denoting the different error terms by $\tilde{\Sigma}_{err,j}$ and the main part by $T_1^{+}$ we can conclude that $$N_{a,b,c} = \min \left\lbrace n \in \N \left| T_{1}^{+}(a,b,c,n) - \sum_{j} \tilde{\Sigma}_{err,j}(c,n) > 0 \right. \right\rbrace .$$ 
This finishes the proof of the theorem.  
\end{proof}
For $c<13$ the $S_j$ will give the main contributions to the circle method as noticed in the last theorem. As the sign of $S_j$ depend on the sign of $\tilde{B}_{j,c,k}$ and the $\tilde{B}_{j,c,k}$ oscillate we will have the following
\begin{corollary}\label{rest}
For $n > \tilde{N}_{a,b,c}$ where $\tilde{N}_{a,b,c}$ is an explicit constant we have
\begin{enumerate}
\item If $ 0 \leq a < b \leq 2$, then the difference $M(a,5,5n+d)-M(b,5,5n+d)$ is
$$
\left\{ 
\begin{array}{ll}
<0&\text{if }  (a,b,d) \in \left\{(0,b,1),(0,2,2),(1,2,2),(1,2,3) \right\},\\[1 ex]
>0&\text{if }  (a,b,d) \in \left\{(0,b,0),(1,2,1), (0,1,3) \right\}.
\end{array}
\right.
$$
\item  If $ 0 \leq a < b \leq 3$, then the difference $M(a,7,7n+d)-M(b,7,7n+d)$ is
$$
\left\{ 
\begin{array}{ll}
<0&\text{if }  (a,b,d) \in \left\lbrace
(0,1,1),(0,1,6),(0,2,1),(0,2,2),(0,3,1),(0,3,6),  \right. \\[1 ex]
&
\left.
(1,2,2),(1,2,4),(1,3,3),(1,3,4),(2,3,3),(2,3,6) 
\right\rbrace \\ [1 ex]
 >0&\text{if }  (a,b,d) \in \left\lbrace
(0,1,0),(0,1,3),(0,1,4),(0,2,0),(0,2,3),(0,3,0),
\right.\\[1ex]
&
\left.
(1,2,1),(1,2,6), (1,3,1), (2,3,2)  \right\rbrace.
\end{array}
\right.
$$
\item  If $ 0 \leq a < b \leq 4$, then the difference $M(a,9,3n+d)-M(b,9,3n+d)$ is
$$
\left\{ 
\begin{array}{ll}
<0&\text{if }  (a,b,d) \in \left\{
(0,1,1),(0,1,6),(0,1,8),(0,2,1),(0,2,2),(0,2,6) \right\},\\[1 ex]
&
\left.
(0,3,1),(0,3,3),(0,3,6),(0,4,1),(0,4,6) ,(0,4,8)
\right. \\ [1 ex]
&
\left.
(1,2,2),(1,2,4),(1,2,7),(1,3,2),(1,3,3),(1,3,4) 
\right. \\ [1 ex]
&
\left.
(1,3,5),(1,3,7),(1,4,4),(1,4,7),(2,3,1),(2,3,3) 
\right.\\ [1 ex]
&
\left.
(2,3,5),(2,3,7),(2,3,8),(2,4,5),(2,4,8),(3,4,0) 
\right. \\ [1 ex]
&
\left.
(3,4,4),(3,4,6),(3,4,8) 
\right\rbrace, \\ [1 ex]
 >0&\text{if }  (a,b,d) \in \left\{
(0,1,0),(0,1,2),(0,1,3),(0,1,4),(0,1,5),(0,1,7), 
\right.\\[1ex]
&
\left.
(0,2,0),(0,2,3),(0,2,4),(0,2,5),(0,2,7),(0,2,8),  
\right. \\ [1 ex]
&
\left.
(0,3,0),(0,3,4),(0,3,7),(0,4,0),(0,4,2),(0,4,3),
\right. \\ [1 ex]
&
\left.
(0,4,4),(0,4,5),(0,4,7),(1,2,1),(1,2,5),(1,2,8), 
\right. \\ [1 ex]
&
\left.
(1,3,0),(1,3,1),(1,3,6),(1,3,8),(1,4,1),(2,3,0),
\right. \\ [1 ex]
&
\left.
(2,3,2),(2,3,4),(2,3,6),(2,4,2),(3,4,1),(3,4,2),
\right. \\ [1 ex]   
&
\left.
(3,4,3),(3,4,5),(3,4,7) 
\right\rbrace.
\end{array}
\right.
$$
\item  If $ 0 \leq a < b \leq 5$, then the difference $M(a,11,11n+d)-M(b,11,11n+d)$ is 
$$
\left\{ 
\begin{array}{ll}
<0&\text{if }  (a,b,d) \in \left\{
(0,1,1),
(0,1,7),
(0,1,8),
(0,1,9), 
(0,2,1),
(0,2,2), \right.\\[1 ex]
&
\left.
(0,2,9),
(0,3,1),
(0,3,8), 
(0,3,9), 
(0,4,1),
(0,4,7),
 \right.
\\ [1 ex]
&
\left.
(0,4,8),
(0,5,1), 
(0,5,9),
(1,2,2), 
(1,2,4), 
(1,3,3), \right.\\ [1 ex]
&
\left.
(1,4,4),
(2,3,3),
(2,3,5), 
(2,3,8),
(2,4,8),
(3,4,4),
\right.\\ [1 ex]
&
\left.
(3,4,7),
(3,4,10),
(3,5,10),
(4,5,5),
(4,5,9)
\right\rbrace, \\ [1 Ex]
 >0&\text{if }  (a,b,d) \in \left\{
(0,b,0),(0,1,3),
(0,1,4),
(0,2,3),(0,2,5),(0,3,4),
\right.\\[1ex]
&
\left.
(0,3,10),
(0,4,3),
(0,4,5),(0,5,3),
(0,5,4),(1,2,1),
\right.\\ [1 ex]
&
\left.
(1,2,5),
(1,2,7),(1,2,8),
(1,3,1),
(1,3,7),
(1,3,10),
\right.\\ [1 ex]
&
\left.
(1,4,1),
(1,4,5),(1,4,9),
(1,5,1),
(1,5,7),
(1,5,8),
\right.\\ [1 ex]
&
\left.
(2,3,2),
(2,3,4),(2,3,10),(2,4,2),
(2,4,9),(2,5,2),(2,5,4),
\right.\\ [1 ex]
&
\left.
(3,4,3),
(3,4,5),
(3,4,9),
(3,5,3),
(3,5,8),(4,5,4) ,
(4,5,7),
\right.\\ [1 ex]
&
\left.
(4,5,8)
\right\rbrace. \\ [1 ex]
\end{array}
\right.
$$
\end{enumerate} 
\end{corollary}
\begin{proof}
The proof uses computer techniques. As $c$ is odd and less than $13$, the inequalities are easily checked by hand using MAPLE. That is done by assuming $c<11$ and $k=c$, because this yields the largest argument in the hyperbolic sine in $S_j$ and from that we only have to compute which sign $\sum_{j} \rho_{j}(a,b,c)\widetilde{B}_{j,c,c}(-n,0)$ has to see which inequality the crank differences obey. For $c=11$ the arguments of the hyperbolic sines could match and cancellation between $S_j$ and $T_j$ can occur. So we have to add to $\sum_{j} \rho_{j}(a,b,c)\widetilde{B}_{j,c,c}(-n,0)$ also $\rho_{1}(a,b,11)\sin\left(\frac{\pi}{11}\right)$ corresponding to the maximal argument in the hyperbolic sine coming from the combination $k=1$, $j=1$, $r=0$ to see which inequality the crank differences obey. Computing all the signs gives the complete list. Two things should be mentioned. The largest argument occurs if $c = k $  and that avoids problems in the computation of the $\widetilde{B}_{j,c,c}(-n,0)$ because $c=9$ is not prime. The other important fact is that we have to modify the constant. As the $S_j$ are no error terms for $c< 11$ the $T_1^{+}(a,b,c,n)$ can be bounded by the error term $\Sigma_{err}^{new}(a,b,c,n):=\frac{2}{\rho(a,b,c)}T_1^{+}(a,b,c,n)$. So we obtain
\be 
\tilde{N}_{a,b,c} = \min \left\lbrace n \in \N \left| \sum_j S_j(a,b,c,n) - \sum_{j}\Sigma_{err,j} - \Sigma_{err}^{new} >0 \right. \right\rbrace. 
\ee 
where the $\Sigma_{err,j}$ are all the error terms of Theorem \ref{blabla} except for the $S_j$-terms.
\end{proof}

\end{document}